\documentclass[12pt]{article}
\usepackage{amsmath,amssymb,amsfonts,amsthm,url}
\usepackage[margin=1in]{geometry}
\usepackage{verbatim}

\expandafter\def\expandafter\UrlBreaks\expandafter{\UrlBreaks
  \do\a\do\b\do\c\do\d\do\e\do\f\do\g\do\h\do\i\do\j
  \do\k\do\l\do\m\do\n\do\o\do\p\do\q\do\r\do\s\do\t
  \do\u\do\v\do\w\do\x\do\y\do\z\do\A\do\B\do\C\do\D
  \do\E\do\F\do\G\do\H\do\I\do\J\do\K\do\L\do\M\do\N
  \do\O\do\P\do\Q\do\R\do\S\do\T\do\U\do\V\do\W\do\X
  \do\Y\do\Z}
  
\usepackage{graphicx}
\usepackage{subcaption}

\theoremstyle{plain}

\newtheorem{theorem}{Theorem}[section]
\newtheorem{proposition}[theorem]{Proposition}
\newtheorem{corollary}[theorem]{Corollary}
\newtheorem{lemma}[theorem]{Lemma}

\theoremstyle{definition}

\newtheorem{construction}[theorem]{Construction}

\newtheorem{definition}[theorem]{Definition}

\usepackage{chngcntr}
\counterwithin{table}{section}

\theoremstyle{remark}

\newtheorem*{remark}{Remark}

\title{Face distributions of embeddings of complete graphs}
\author{Timothy Sun\\Columbia University}
\date{}

\newcommand{\Z}{\mathbb{Z}}

\begin{document}

\maketitle

\begin{abstract}
A longstanding open question of Archdeacon and Craft asks whether every complete graph has a minimum genus embedding with at most one nontriangular face. We exhibit such an embedding for each complete graph except $K_8$, the complete graph on 8 vertices, and we go on to prove that no such embedding can exist for this graph. Our approach also solves a more general problem, giving a complete characterization of the possible face distributions (i.e. the numbers of faces of each length) realizable by minimum genus embeddings of each complete graph. We also tackle analogous questions for nonorientable and maximum genus embeddings.  
\end{abstract}

\section{Introduction}

The celebrated Map Color Theorem of Ringel and Youngs~\cite{RingelYoungs} boils down to the fact that for $n \geq 3$, $K_n$, the complete graph on $n$ vertices, can be embedded in a sphere with

$$I(n) = \left\lceil \frac{(n-3)(n-4)}{12} \right\rceil$$

handles. Equivalently, $K_n$ can be embedded in the orientable surface of genus $I(n)$. These embeddings are provably minimal in terms of genus, as they match a lower bound given by the Euler polyhedral equation. Starting with the work of Lawrencenko \emph{et al.}~\cite{LNW}, one direction of continued research on this topic examines the number of essentially different \emph{minimum genus} embeddings, and several different approaches (see, e.g., Bonnington \emph{et al.}~\cite{Bonnington-Exponential}, Korzhik and Voss~\cite{Korzhik-Exponential}, Goddyn \emph{et al.}~\cite{Goddyn-Exponential}) have yielded families of embeddings whose sizes are exponential in the number of vertices. 

The main combinatorial technique for finding such embeddings are current graph constructions, where an embedding of a smaller, edge-labeled graph can be used to generate a highly symmetric embedding of a much larger graph. The proof generally proceeds in two steps: 

\begin{itemize}
\item A \emph{regular} step which involves finding a suitable current graph for triangularly embedding a graph that is close to complete (e.g. an embedding of a complete graph minus three edges).
\item An \emph{additional adjacency} step which modifies the embedding and the graph so it becomes complete (e.g. using a handle to add the three missing edges).
\end{itemize}

When $n \equiv 0, 3, 4, 7 \pmod{12}$, the embedding is a triangulation of the surface and no additional adjacency step is necessary. For the other cases, Korzhik and Voss~\cite{KorzhikVoss} exhibit exponentially many embeddings by modifying (in most cases) the regular step of the proof found in Ringel~\cite{Ringel-MapColor}. Their proof that the different embeddings are nonisomorphic involves showing that an isomorphism of nontriangular faces cannot be extended to the whole embedding. 

We exhibit new embeddings of complete graphs that are not isomorphic for a more fundamental reason: the distributions of the face lengths are different. Archdeacon and Craft~\cite{Archdeacon} ask whether or not every complete graph has a minimum genus embedding that is \emph{nearly triangular}, one where at most one face is nontriangular. We construct such an embedding for every complete graph except $K_8$ and prove that $K_8$ has no such embedding. We note that the for most of the complete graphs, the original constructions did not produce nearly triangular embeddings (see the exposition in Korzhik and Voss~\cite{KorzhikVoss}).

One can also ask if there are minimum genus embeddings which manifest all other possible combinations of nontriangular faces (e.g. two quadrangular faces), as permitted by the Euler polyhedral equation. Besides the aforementioned $K_8$, it turns out that the only other complete graph which does not realize all its predicted embedding types is $K_5$. The results in this paper can thus been seen as a step in understanding the embedding polynomials (as introduced by Gross and Furst~\cite{GrossFurst}) of the complete graphs---we fully determine which coefficients corresponding to minimum genus embeddings are nonzero.

In Sections 2-4, we review some background on topological graph theory and current graphs. We prove the main result across Sections 5-12, where the different cases are handled in roughly increasing difficulty of the additional adjacency problem. Some variations of the original problem are solved in Sections 13 and 14, and some potential future directions are outlined in Section 15.

\section{Notation and terminology}

For a comprehensive background on topological graph theory, see Gross and Tucker~\cite{GrossTucker}. For a complete proof of the Map Color Theorem, see Ringel~\cite{Ringel-MapColor}. 

In this paper, a \emph{graph} $G = (V,E)$ consists of a finite set of vertices $V$ and a set of (unoriented) edges $E$. We regard the vertices and edges as a ``1-dimensional cell complex'' where vertices are points and edges are arcs connecting two points. In this paper, we mostly consider \emph{simple} graphs, those without parallel edges between two vertices or self-loops. Vertices and edges are written as letters or numbers. When the graph is simple, we will also write edges as pairs of vertices $(u,v)$, and we say that $u$ and $v$ are adjacent, and that $u$ is a neighbor of $v$ and vice versa. Given an edge $e$, there are two \emph{edge ends} $e^+$ and $e^-$, each incident with a vertex. The \emph{degree} of a vertex is the number of edge ends incident with it (in particular, a self-loop contributes 2 to the degree of a vertex). A \emph{directed graph} $D = (V, E)$ consists of vertices and a set of \emph{arcs} $E$, which are edges with specified orientations. 

Except in Sections \ref{sec-nonorient} and \ref{sec-maxgenus}, we focus solely on embeddings in orientable surfaces. An \emph{embedding} of a graph $G$ on a surface $S$ is an injective map $\phi: G \to S$. We only consider \emph{cellular} embeddings, those where $S \setminus \phi(G)$ decomposes into a disjoint union of open disks (denoted $F$), which we call \emph{faces}. The boundary of each face coincides with a sequence of \emph{corners}, which consist of a vertex and a pair of incident edge ends. To describe a face, it often suffices to give a cyclic ordering of (possibly nondistinct) vertices $[v_1, v_2, \dotsc, v_k]$. We say that a face is \emph{$k$-sided} or is of \emph{length} $k$, where $k$ is the number of elements in the cyclic ordering. We sometimes call a $3$-sided face a \emph{triangle}, a $4$-sided face a \emph{quadrilateral}, and so on.

We write $K_n$ to denote the \emph{complete graph} on $n$ vertices: the simple graph where every pair of vertices is connected by an edge. If $H$ is a subgraph of $G$, $G-H$ is the graph where we take $G$ and delete the edges of $H$. Typically, we take $G$ to be a complete graph, so by symmetry, we do not need to explicitly specify the inclusion map $H \to G$. 

\section{Combinatorics of graph embeddings}

A \emph{rotation} at vertex $v$ is a cyclic permutation of the edge ends incident with $v$. A \emph{rotation system} $\Phi$ of a graph $G$ is a collection of rotations for each vertex of $G$. In the case of a simple graph, we only need to specify a cyclic ordering of the neighbors of $v$. Rotation systems of simple graphs are often written as a table of symbols, so we sometimes refer to the rotation at $v$ as \emph{row $v$}. For an embedding $\phi: G \to S$ in an orientable surface, we can obtain a rotation system by considering the clockwise order of edges incident with each vertex, for some orientation of the surface $S$. The Heffter-Edmonds principle states that this is actually a one-to-one correspondence---each rotation system induces an embedding that is unique up to ``homeomorphism of pairs.'' 

The surface can be constructed in a group-theoretic way. Consider the involution $\theta: e^+\mapsto e^-$ for all edges $e$ and regard $\Phi$ as a permutation of the set of edge ends. Then, the cycles of the composition $\Phi \hspace{1pt}\circ\hspace{1pt} \theta$ define the boundaries of the faces. Intuitively, this permutation is essentially tracing around the boundaries of the faces. In all our drawings, we take the convention where rotations are specified in clockwise order, which induces a counterclockwise orientation on the faces. 

The Heffter-Edmonds principle not only reduces the problem of finding embeddings to a purely combinatorial one, but also shows that there are only finitely many homeomorphism types of embeddings of a given graph. The $\emph{genus}$ of an embedding $\phi: G \to S$ is just the genus of $S$, and the \emph{minimum genus} $\gamma(G)$ of $G$ is the smallest genus over all embeddings $\phi$. 

Let $S_g$ denote the surface of genus $g$, i.e., a sphere with $g$ handles. Given an embedding $\phi: G \to S_g$, the fundamental equation governing cellular embeddings is the \emph{Euler polyhedral equation}
$$|V(G)| - |E(G)| + |F(G)| = 2 - 2g.$$
For a fixed graph $G$, the genus of the surface $S$ it is embedded in is then intimately related to the number of faces. A embedding of a graph is said to be \emph{triangular} if all its faces are triangular, and \emph{nearly triangular} if at most one face is not triangular. For a simple graph, a triangular embedding maximizes the number of faces and thus it has minimum genus.  The starting point of the Map Color Theorem and many other graph embedding problems is the following refinement of the Euler polyhedral equation:

\begin{proposition}
If a graph $G$ has a triangular embedding in $S_g$, then number of edges in $G$ is
$$|E(G)| = 3|V(G)|-6+6g.$$
\label{prop-trianglebound}
\end{proposition}

From this relationship, we can figure out what types of minimum genus embeddings are permissible under the Euler polyhedral equation. Some complete graphs, $K_7$ for example, do have a triangular embedding in some surface, but others do not. However, we can triangulate the nontriangular faces with additional edges without increasing the genus. Substituting ``$K_n$ plus $t$ edges'' into Proposition~\ref{prop-trianglebound} yields
$$6n + 12g - 12 = 2(|E(K_n)|+t) = n(n-1)+2t.$$
To remove the genus parameter $g$, we take the resulting equation modulo $12$:
$$2t \equiv -(n-3)(n-4) \pmod{12}.$$
Like in the Map Color Theorem, the analysis now breaks down into twelve Cases (with a capital ``C'') depending on the residue $n \bmod 12$. Initially, we observe the following:

\begin{itemize}
\item If $n \equiv 0, 3, 4, 7 \pmod{12}$, $t \equiv 0 \pmod{6}$.
\item If $n \equiv 2, 5 \pmod{12}$, $t \equiv 5 \pmod{6}$.
\item If $n \equiv 1, 6, 9, 10 \pmod{12}$, $t \equiv 3 \pmod{6}$.
\item If $n \equiv 8, 11 \pmod{12}$, $t \equiv 2 \pmod{6}$.
\end{itemize}

One way of stating the Map Color Theorem is to say that there exist triangular embeddings where we actually have equality for the number of extra edges $t$. 

The goal of the present paper is to classify the different possible face distributions for, primarily, minimum genus embeddings of complete graphs. The \emph{face distribution}\footnote{White~\cite{White-GraphGroups} refers to this as the \emph{region distribution}.} of an embedding is the sequence $f_1, f_2, \dotsc$ where $f_i$ is the number of faces of length $i$. For example, a minimum genus embedding of $K_7$ triangulates the torus, so its face distribution is 
$$0, 0, 14, 0, 0, \dotsc$$ 
For the residue classes $n \not\equiv 0, 3, 4, 7 \pmod{12}$, we try to partition the $t$ ``chordal'' edges into the faces to get embeddings for each possible face distribution permitted by the Euler polyhedral equation. For example, if $n = 14$, then $t=5$. As seen in Figure~\ref{fig-chord}, distributing all five additional edges into the same face gives us an 8-sided face, but we could distribute the edges in a different way to get one 6-sided face and one 5-sided face. 

\begin{figure}[!ht]
\centering
\includegraphics{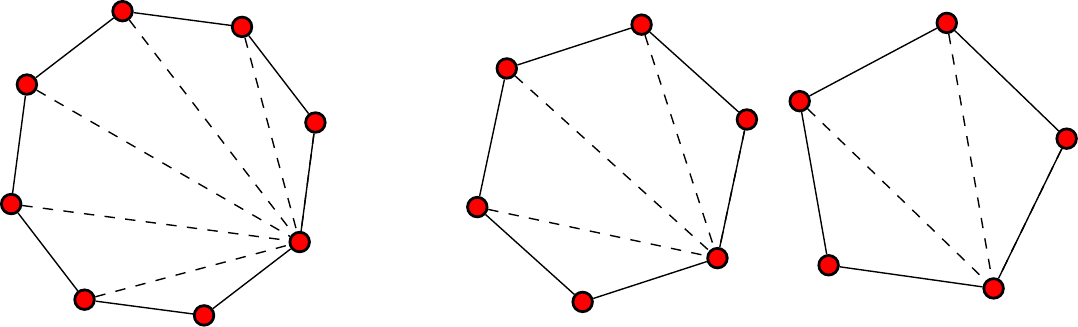}
\caption{For a minimum genus embedding, the missing chords needed to make the embedding triangular could be distributed among faces in a few different ways.}
\label{fig-chord}
\end{figure}

Instead of writing out face distributions in full and counting all the triangular faces, we say that an embedding is of \emph{type $(a_1, \dotsc, a_i)$}, if it has faces of length $a_1, a_2, \dotsc, a_i$, where $a_1 \geq a_2 \geq \dots \geq a_i > 3$ and all the other faces are triangular. In this terminology, $K_{14}$ could have embeddings of type $(8)$ and $(6, 5)$. In general, if $t = b_1 + \dots + b_j$ is a partition of $t$ into positive integers $b_i$, we are looking for an embedding of type $(b_1 + 3, b_2 + 3, \dotsc, b_j + 3)$. Thus, we need to find the following embedding types:

\begin{itemize}
\item For $n \equiv 2, 5 \pmod{12}$, types $(8)$, $(7,4)$, $(6,5)$, $(6,4,4)$, $(5,5,4)$, $(5,4,4,4)$, and 

$(4,4,4,4,4)$.
\item For $n \equiv 1, 6, 9, 10 \pmod{12}$, types $(6)$, $(5,4)$, $(4,4,4)$.
\item For $n \equiv 8, 11 \pmod{12}$, types $(5)$ and $(4,4)$.
\end{itemize}

We appeal to Proposition~\ref{prop-trianglebound} to show that regardless of the graph, embeddings of these types are minimal.

\begin{proposition}
Suppose there exists an embedding $\phi$ of a simple graph $G$ of type $(a_1, \dotsc, a_i)$, where $(a_1-3) + \dots + (a_i-3) \leq 5$. Then $\phi$ is a minimum genus embedding.
\end{proposition}
\begin{proof}
The inequality is equivalent to the statement that there are at most 5 extra edges. Proposition~\ref{prop-trianglebound} loosely states that each handle allows for 6 extra edges, so the number of edges of $G$ exceeds the number of edges in a triangular embedding in any surface of smaller genus. 
\end{proof}

We state our main result in this language.

\begin{theorem}
For all $n \geq 3$, $n \neq 5, 8$ and for every partition of $t = t(n)$ into positive integers
$$t = b_1 + b_2 + \dots + b_j,$$
for $b_1 \geq b_2 \geq \dots \geq b_j$, there exists an embedding of type $(b_1+3, b_2+3, \dotsc, b_j+3)$ of $K_n$. $K_5$ only has minimum genus embeddings of type $(8)$, $(7,4)$, $(6,4,4)$, $(5,5,4)$, and $(4,4,4,4,4)$, and $K_8$ only has minimum genus embeddings of type $(4,4)$. 
\end{theorem}

Our result answers the original question of Archdeacon and Craft~\cite{Archdeacon}, showing that

\begin{corollary}
For $n \geq 3$, $n \neq 8$, there exists a nearly triangular minimum genus embedding of $K_n$.
\end{corollary}

For the Cases where $t = 2$ or $3$, it turns out that practically all of the difficulty is in finding the nearly triangular embedding, i.e. the embeddings of types $(5)$ and $(6)$. Using those embeddings, it is straightforward to obtain the other types. We say a face is \emph{simple} if it is not incident with the same vertex more than once. 

\begin{lemma}
Let $G$ be a simple graph with minimum degree $2$. For any orientable embedding of $G$, all 5-sided faces are simple. All 6-sided faces have at most one repeated vertex---in particular, it is of the form $[a, b, x, c, d, x']$, where only $x$ and $x'$ are possibly nondistinct.
\label{lem-5simp}
\end{lemma}
\begin{proof}
Suppose some vertex $v$ appears twice in some $5$-sided face. The face cannot be of the form $[\dots v, v \dots]$, otherwise there would be a self-loop at $v$. On the other hand, the face also cannot be of the form $[\dots v, w, v \dots]$ for some vertex $w$, because otherwise $w$ would have degree 1, or there would be more than one edge incident with $v$ and $w$. 

By the same reasoning, the two instances of a repeated vertex on a 6-sided face must appear ``opposite'' each other. Suppose two vertices $a$ and $b$ appeared twice on the same face. Without loss of generality, the face must be of the form $[a, b, c, a, b, c']$. However, this would imply that the embedding is on a nonorientable surface, since the edge $(a,b)$ is traversed twice in the same direction.\footnote{The union of the face and the edge $(a,b)$ is homeomorphic to a Mobius band.}
\end{proof}

\begin{proposition}
If $K_n$ has an orientable embedding of type $(5)$ (resp. type $(6)$), then it has an embedding of type $(4,4)$) (resp. types $(5,4)$ and $(4, 4, 4)$). 
\label{prop-simple6}
\end{proposition}
\begin{proof}
In the embedding of type $(5)$, the 5-sided face $f$ is simple by Lemma~\ref{lem-5simp}, so if $f$ is of the form $[\dots a, b, c \dots]$, $a$ must be different from $c$, and the edge $(a,c)$ is not incident with this face. If we delete the edge $(a,c)$ and add it back in as a chord of $f$, we get an embedding of type $(4,4)$. 

Applying Lemma~\ref{lem-5simp} again, suppose the 6-sided face in an embedding of type $(6)$ is of the form $[a, v, w, a', x, y]$, where $a$ and $a'$ are possibly not distinct. Like in the previous case, we alter the positions of edges $(v,w)$ and $(x,y)$, like in Figure~\ref{fig-6to444}, so that they become chords. The result is an embedding of type $(4,4,4)$. Applying this procedure to just one of the edges yields an embedding of type $(5,4)$.
\end{proof}

\begin{figure}[!ht]
\centering
\includegraphics{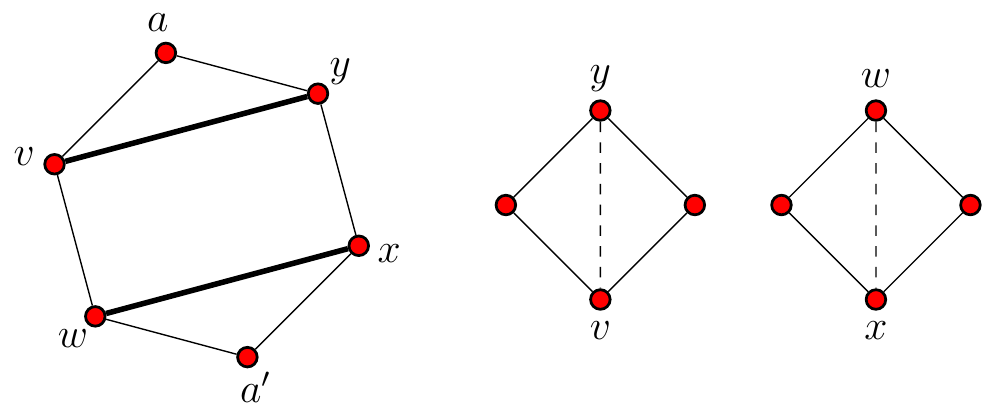}
\caption{Changing an embedding of type $(6)$ into one of type $(4,4,4)$. The dashed and thickened lines represent the old and new locations of the edges, respectively.}
\label{fig-6to444}
\end{figure}

The idea of changing the location of an existing edge to a nontriangular face is prevalent in this paper. We call such an operation a \emph{chord exchange} $\pm(u,v)$ or say that we are \emph{exchanging the chord $(u,v)$}. 

\section{Current graphs}

The main tools for constructing triangulations of large complete graphs are known as \emph{current graphs}. We describe them slightly informally here---a rigorous topological treatment can be found in Gross and Tucker~\cite{GrossTucker}. Let $D$ be a directed graph, possibly with self-loops and parallel edges, with an embedding $\phi: D \to S$ in an orientable surface $S$, and let $\lambda: E(D) \to \Gamma$ be an \emph{assignment}, where arcs are labeled with elements, which we call \emph{currents}, from an abelian group $\Gamma$. The triple $\langle D, \phi, \lambda \rangle$ is called a \emph{current graph}. The groups we consider in this paper are the cyclic groups $\Z_n$, i.e., the integers under addition modulo $n$. For convenience, we sometimes use negative signs to describe vertices, e.g., $-1$ instead of $12s{+}6$ in the group $\Z_{12s+7}$. 

The name of this computational tool comes from the desirable property that at most vertices, ``flow'' is conserved. Note that if an arc is assigned the current $\gamma$, replacing the arc with an arc in the opposite direction with current $-\gamma$ yields an equivalent current graph. The \emph{excess} of a vertex is the sum of the currents of arcs incident with $v$, when oriented towards $v$. We say that a vertex $v$ satisfies \emph{Kirchhoff's current law} (KCL) if its excess is $0$. We call a vertex $v$ a \emph{vortex} if KCL is not satisfied there, and for each corner of a face incident with $v$, we mark it with a letter. Let the \emph{order} of an element $g \in \Z_n$ be the smallest positive integer $p$ such that $pg = 0$. We consider current graphs which have three different types of vortices:

\begin{itemize}
    \item [(T1)] If $v$ is a vortex of degree $1$ and $\Gamma = \Z_n$, its excess has order $n$.
    \item [(T2)] If $v$ is a vortex of degree $1$ and $\Gamma = \Z_{2n}$, its excess has order $n$. 
    \item [(T3)] If $v$ is a vortex of degree $3$ and $\Gamma = \Z_{3n}$, its excess has order $n$, and the currents $\alpha, \beta, \gamma$ flowing into $v$ satisfy either $\alpha,\beta,\gamma \equiv 1 \pmod{3}$ or $\alpha,\beta,\gamma \equiv 2 \pmod{3}$.
\end{itemize}

We say that a current graph $\langle D, \phi, \lambda \rangle$ is \emph{valid} if it satisfies the following ``construction principles'':

\begin{itemize}
    \item [(C1)] Each vertex of $D$ has degree 3 or 1.
    \item [(C2)] $\phi$ is a one-face embedding.
    \item [(C3)] Each element of $\Gamma\setminus\{0\}$ or its inverse appears exactly once as a current.
    \item [(C4)] For each non-vortex $v$ of degree 3, the sum of the inward flowing currents satisfies KCL.
    \item [(C5)] If $\Gamma \cong \Z_{2n}$, the element $n \in \Z_{2n}$ must be assigned to an edge incident with a vertex of degree 1. 
    \item [(C6)] Each vortex is of type (T1), (T2), or (T3).
\end{itemize}

These construction principles guarantee, among other things, that the resulting embedding is triangular. The number of faces in the embedding $\phi$ is referred to as the \emph{index} of the current graph. By principle (C2), we only consider index 1 current graphs, though in two instances ($K_{20}$ and $K_{30}$), we derive our embedding from index 3 current graphs. However, in the interest of brevity, we omit the descriptions of these current graphs and work on the rotation systems directly. 

A standard way of checking if a rotation system is triangular is known as \emph{Rule R*}, which is guaranteed to be satisfied in most cases by KCL.

\begin{definition}
A rotation system satisfies \emph{Rule R*} if for all edges $(i,k)$, if row $i$ is of the form
$i. \, \dots \, j \, k \, l \dots,$
then row $k$ is of the form $k. \, \dots l \, i \, j \dots$
\end{definition}

\begin{theorem}[see Ringel~{\cite[\S5.1]{Ringel-MapColor}}]
A rotation system of $G$ satisfies Rule R* if and only if it describes a triangular embedding of $G$ on an orientable surface.
\end{theorem}

We trace the boundary of the one face and write down the arcs and letters (from vortices) in a cyclic sequence. If we traverse arc $a$ in the same direction as its orientation, we replace it with $\lambda(a)$. Otherwise, we replace it with $-\lambda(a)$. Because of principle (C5), the face boundary will have two consecutive instances of the element of order 2, but we will only record it once.\footnote{Technically, we end up with a disjoint collection of pairs of parallel edges, but we condense each pair into one edge.} To emphasize this omission, we follow the convention where the vertex of degree 1 incident with this arc is not drawn. The resulting cyclic sequence of elements of $\Gamma$ and letters is the \emph{log} of the face boundary. 

Figure~\ref{fig-cur-case2ex} gives an example of a current graph illustrating all the vortex types and construction principles. The rotations at solid vertices are oriented clockwise, and the rotations at hollow vertices are oriented counterclockwise. The log of this face, which essentially describes the rotation at vertex 0, is 
{
\setlength{\arraycolsep}{4.5pt}
$$\begin{array}{rrrrrrrrrrrrrrrrrrrrrrrrrrrrrrrr}
0. & 11 & x & 7 & a & 8 & w & 13 & 1 & 15 & 9 & 6 & 5 & u & 16 & y & 2 & v & 10 & c & 14 & 17 & 12 & 3 & 4 & b 
\end{array}$$
}
\begin{figure}[!ht]
\centering
\includegraphics{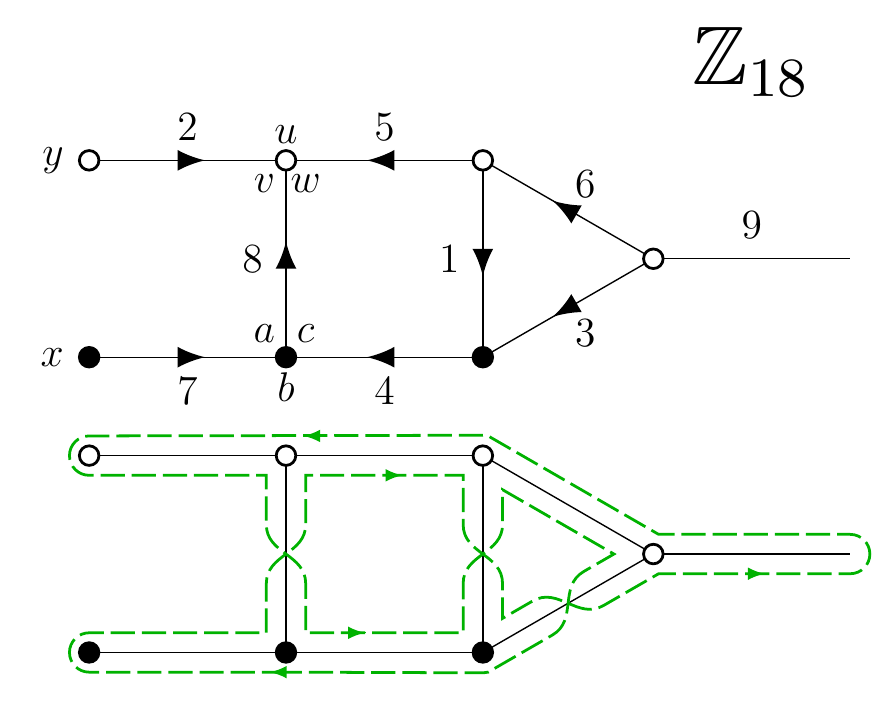}
\caption{A current graph used by Ringel and Youngs~\cite{RingelYoungs-Case2} and the boundary of the single face of the embedding.}
\label{fig-cur-case2ex}
\end{figure}
To generate the remaining rows, we use what is known as the \emph{additive rule}. To find the rotation at vertex $k \in \Gamma$, we do the following:
\begin{itemize}
\item For each entry $k' \in \Gamma$ in the log, increment it by $k$. 
\item For vortex letters $x$ of type (T1), leave it as $x$.
\item For vortex letters $y$ of type (T2), replace it with $y_{k \bmod 2}$. 
\item For vortex letters $a,b,c$ of type (T3), suppose without loss of generality that the log is of the form
{
\setlength{\arraycolsep}{4.5pt}
$$\begin{array}{rrrrrrrrrrrrrrrrrrrrrrrrrrrrrrrr}
0. & \dots & a & \dots & b & \dots & c & \dots
\end{array}$$
}
and that the incoming currents are all congruent to $1 \pmod{3}$. 
\begin{itemize}
\item If $k \equiv 0 \pmod{3}$, keep the letters the same. 
\item If $k \equiv 1 \pmod{3}$, replace them as
{
\setlength{\arraycolsep}{4.5pt}
$$\begin{array}{rrrrrrrrrrrrrrrrrrrrrrrrrrrrrrrr}
k. & \dots & b & \dots & c & \dots & a & \dots
\end{array}$$
}
\item If $k \equiv 2 \pmod{3}$, replace them as
{
\setlength{\arraycolsep}{4.5pt}
$$\begin{array}{rrrrrrrrrrrrrrrrrrrrrrrrrrrrrrrr}
k. & \dots & c & \dots & a & \dots & b & \dots
\end{array}$$
}
\end{itemize}
\end{itemize}
For the numbered vertices, the rotations look like
{
\setlength{\arraycolsep}{4pt}
$$\begin{array}{rrrrrrrrrrrrrrrrrrrrrrrrrrrrrrrr}
0. & 11 & x & 7 & a & 8 & w & 13 & 1 & 15 & 9 & 6 & 5 & u & 16 & y_0 & 2 & v & 10 & c & 14 & 17 & 12 & 3 & 4 & b & \\
1. & 12 & x & 8 & c & 9 & v & 14 & 2 & 16 & 10 & 7 & 6 & w & 17 & y_1 & 3 & u & 11 & b & 15 & 0 & 13 & 4 & 5 & a & \\
2. & 13 & x & 9 & b & 10 & u & 15 & 3 & 17 & 11 & 8 & 7 & v & 0 & y_0 & 4 & w & 12 & a & 16 & 1 & 14 & 5 & 6 & c & \\
3. & 14 & x & 10 & a & 11 & w & 16 & 4 & 0 & 12 & 9 & 8 & u & 1 & y_1 & 5 & v & 13 & c & 17 & 2 & 15 & 6 & 7 & b & \\
4. & 15 & x & 11 & c & 12 & v & 17 & 5 & 1 & 13 & 10 & 9 & w & 2 & y_0 & 6 & u & 14 & b & 0 & 3 & 16 & 7 & 8 & a \\
\vdots &
\end{array}$$
}
and so on. For all the lettered vertices, their rotations are ``manufactured'' so that Rule R* is satisfied. For example, we obtain the rows
{
\setlength{\arraycolsep}{4.2pt}
$$\begin{array}{rrrrrrrrrrrrrrrrrrrrrrrrrrrrrrrr}
a. & 0 & 7 & 11 & 3 & 10 & 14 & 6 & 13 & 17 & 9 & 16 & 2 & 12 & 1 & 5 & 15 & 4 & 8 & \\
x. & 0 & 11 & 4 & 15 & 8 & 1 & 12 & 5 & 16 & 9 & 2 & 13 & 6 & 17 & 10 & 3 & 14 & 7 & \\
y_0. & 0 & 16 & 14 & 12 & 10 & 8 & 6 & 4 & 2 & \\
\end{array}$$
}
If the current graph is valid, we get a triangular embedding of a graph with $|\Gamma|$ numbered vertices, all pairwise adjacent, and some lettered vertices, all pairwise nonadjacent. Vortices of type (T1) and (T3) are adjacent to all the numbered vertices, while vortices of type (T2) each split into two vertices that are adjacent to half of the numbered vertices. 

The ``geometry'' of the current graphs we will encounter contain ladder-like subgraphs, like in the middle of the current graph in Figure~\ref{fig-cur-case2ex}. The additional adjacency steps only use part of a current graph, so we ignore the unneeded parts by replacing ladders with boxes, as in Figure~\ref{fig-ladder}. All the vertices replaced by the box satisfy KCL, and the currents are assigned such that construction principle (C3) holds. In this paper, the rotations have already been specified, but the originators of this notation, Korzhik and Voss~\cite{KorzhikVoss}, used the box to mean any set of rotations that produce a one-face embedding. Additionally, we may omit some current assignments on edges outside of these boxes for clarity. Typically they can be recovered by KCL.

\begin{figure}[!ht]
\centering
\includegraphics{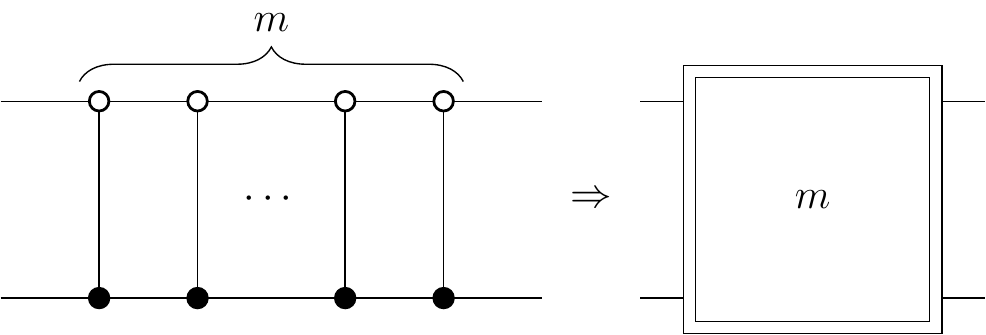}
\caption{Instead of drawing the ladder subgraph on the left, we replace it with a box indicating the number of ``rungs.''}
\label{fig-ladder}
\end{figure}

\section{Cases 2 and 5}

For $n \equiv 2, 5 \pmod{12}$, we expect to find a nearly-triangular embedding with an 8-sided face. Fortunately, we can leverage existing constructions for these Cases:

\begin{theorem}[Jungerman~\cite{Jungerman-KnK2}, Ringel~{\cite[p.83]{Ringel-MapColor}}]
For $s \geq 1$, there exists a triangular embedding of $K_{12s+2}-K_2$.
\label{thm-case2}
\end{theorem}
\begin{theorem}[Youngs~\cite{Youngs-3569} or Ringel~{\cite[\S 9.2]{Ringel-MapColor}}]
For $s \geq 0$, there exists a triangular embedding of $K_{12s+5}-K_2$.
\end{theorem}

From one of these embeddings, arbitrarily adding the missing edge causes two triangular faces to combine into an 8-sided face. Figure~\ref{fig-8handle} shows this operation along with how the orientation of the two participating faces affect the final nontriangular face. Achieving the other face distributions requires a few small modifictions. We prove the following using those embeddings:

\begin{figure}[!ht]
\centering
\includegraphics{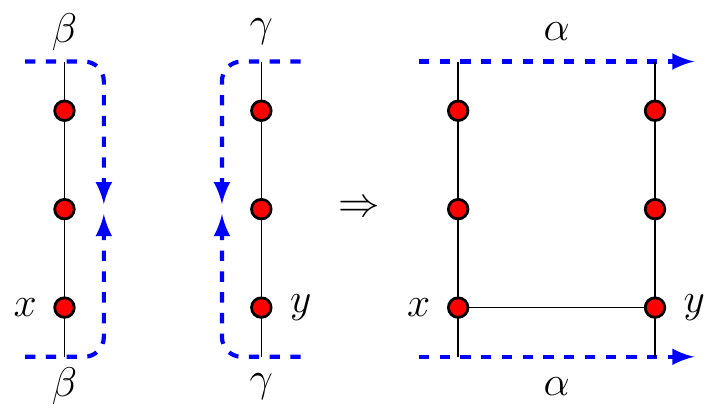}
\caption{Adding an edge with the help of one handle merges two triangular faces together. The pairs of thick dashed arrows labeled with the same letters are identified together.}
\label{fig-8handle}
\end{figure}

\begin{proposition}
For $s \geq 1$, there exists embeddings of type $(8)$, $(7,4)$, $(6,5)$, $(6,4,4)$, $(5,5,4)$, $(5,4,4,4)$, and $(4,4,4,4,4)$ of $K_{12s+2}$ and $K_{12s+5}$.
\label{prop-knk2}
\end{proposition}

\begin{proof}
Let $x$ and $y$ be the two nonadjacent vertices. The general approach is to exchange chords in the 8-sided face, which does not increase the genus of the embedding. Some of these constructions are illustrated in Figure~\ref{fig-5444-65}.

(Types $(7, 4)$, $(6, 4, 4)$, $(5, 5, 4)$, and $(4, 4, 4, 4, 4)$) Since $s \geq 1$, $x$ and $y$ have at least 12 neighbors. Because there are many neighbors, we can find faces $[x, a, b]$ and $[y, c, d]$ such that $a, b, c, d$ are all distinct vertices. After attaching these two faces with a handle and adding the edge $xy$, the resulting $8$-sided face will be $[x, a, b, x, y, c, d, y]$. Exchanging the following sets of chords yields the following embeddings:

\begin{itemize}
\item type $(7,4)$: $\pm(a,y)$,
\item type $(6,4,4)$: $\pm(a,d)$,
\item type $(5,5,4)$: $\pm(a,c)$,
\item type $(5,4,4,4)$: $\pm(a,d), \pm(b,y)$, and 
\item type $(4,4,4,4,4)$: $\pm(a,d), \pm(b,c)$.
\end{itemize}

\begin{figure}[!ht]
\centering
\includegraphics{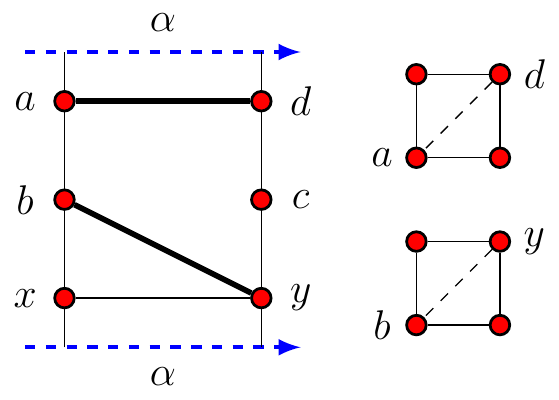}
\quad \quad \quad
\includegraphics{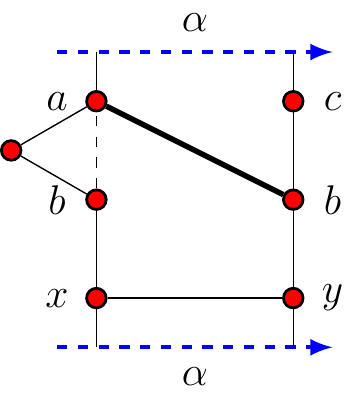}
\caption{Finding embeddings of types $(5,4,4,4)$ and $(6,5)$.}
\label{fig-5444-65}
\end{figure}

(Type $(6,5)$) We assert that there exist faces $[x, a, b]$, $[y, b, c]$, where $a \neq c$. Since $s \geq 1$, vertex $b$ has at least 13 neighbors. Without loss of generality, we may assume that the rotation at $b$ is of the form 
$$\begin{array}{rrrrrrrrrrrrrrrrrrrrrrrrrrrrrrrr}
b. & \dots & y & c & \dots & a & x & \dots,
\end{array}$$
where there are at least two other vertices in between $y$ and $x$ in the cyclic sequence. Hence, these triangles incident with $b$ are the desired faces. Adding the edge $(x,y)$ using those two faces and exchanging the chord $(a,b)$ yields an embedding of type $(6, 5)$. 
\end{proof}

\begin{remark}
One might ask why we need $a \neq c$ for the type $(6,5)$ construction. If they are the same vertex, then the edge $(a,b)$ appears twice on the $8$-sided face. Deleting that edge causes the genus to decrease and the face to split in two, invalidating the ``locally planar'' intuition that our drawings are based on. 
\end{remark}

We note that $K_5$, despite there being a triangular embedding of $K_5-K_2$, does not realize all its predicted face distributions. An exhaustive enumeration produced the following:

\begin{proposition}[see Gagarin \emph{et al.}~\cite{Gagarin-Torus} or White~{\cite[p.270]{White-GraphGroups}}]
$K_5$ has embeddings of type $(8)$, $(7,4)$, $(6,4,4)$, $(5,5,4)$, and $(4,4,4,4,4)$, but no embeddings of type $(6,5)$ or $(5,4,4,4)$.
\end{proposition}

It can be verified that the constructions in Proposition~\ref{prop-knk2} for the latter two cases cannot be applied to the essentially unique planar embedding of $K_5-K_2$. Each vertex has too few neighbors. 

\section{Case 9}

In the previous section, we found nearly triangular embeddings by taking a triangular embedding and adding a single edge. Jungerman's solution for Case 9 also has a simple additional adjacency solution that involves only one extra edge. We say that $G_n$ is a \emph{split-complete graph} if we can label its vertices $1, 2, \dotsc, n-1, x_0, x_1$ such that 

\begin{itemize}
\item $1, \dotsc, n-1$ are all pairwise adjacent, and
\item the neighbors of $x_0$ and the neighbors of $x_1$ form a partition of $\{1, \dotsc, n-1\}$.
\end{itemize}

The aforementioned solution of Jungerman employed a beautiful construction for split-complete graphs.

\begin{theorem}[see Ringel~{\cite[\S6.5]{Ringel-MapColor}}]
For $s \geq 0$, there exists a triangular embedding of a split-complete graph $G_{12s+9}$.
\label{thm-case9graphs}
\end{theorem}

In the proof of the Map Color Theorem, embeddings were expressed in dual form, where the vertices were regarded as ``countries'' drawn on surfaces. The countries $x_0$ and $x_1$ were then connected with a handle and then merged into one ``cylindrical region.'' Upon closer examination, the resulting embedding in primal form is in fact nearly triangular.

\begin{proposition}
If there exists a triangular embedding of a split-complete graph $G_n$, then there exists an embedding of type $(6)$ of $K_n$.
\label{prop-splitcomplete}
\end{proposition}

\begin{proof}
Add the edge $x_0$ and $x_1$ arbitrarily as we did for Cases 2 and 5. Note that the newly added edge $(x_0,x_1)$ appears twice in the resulting 8-sided face. Locally contracting this edge leaves a 6-sided face, as in Figure~\ref{fig-splitcontract}.
\end{proof}

\begin{figure}[!ht]
\centering
\includegraphics{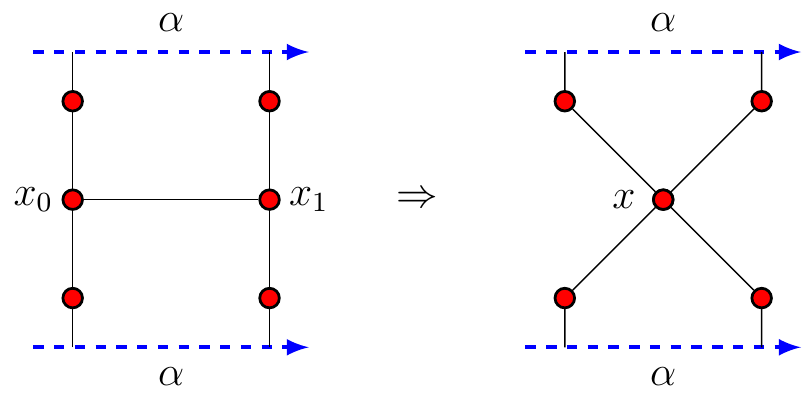}
\caption{Adding a handle to add an edge, and then contracting it to get a 6-sided face.}
\label{fig-splitcontract}
\end{figure}

\begin{corollary}
For $s\geq 0$, there exists a nearly triangular minimum genus embedding of $K_{12s+9}$. 
\end{corollary}
\begin{corollary}
For $s\geq 0$, there exist embeddings of type $(6)$, $(5,4)$, and $(4,4,4)$ of $K_{12s+9}$. 
\end{corollary}

\section{Case 6}

\begin{theorem}
For $s\geq 0$, there exists a nearly triangular minimum genus embedding of $K_{12s+6}$. 
\end{theorem}
\begin{proof}
For $s=0$, such an embedding can be found by deleting a vertex from the triangular embedding of $K_7$ in the torus. For $s=1$, Mayer~\cite{Mayer-Orientables} constructed a split-complete $G_{18}$, so applying Proposition~\ref{prop-splitcomplete} yields the desired embedding. The larger-order cases are covered by combining triangular embeddings of $K_{12s+6}-P_3$ (Proposition~\ref{prop-30} for $s=2$, Theorem~\ref{thm-gross} for $s \geq 3$), with Lemma~\ref{lem-addpath}.
\end{proof}
\begin{corollary}
For $s\geq 0$, there exist embeddings of type $(6)$, $(5,4)$, and $(4,4,4)$ of $K_{12s+6}$. 
\end{corollary}

The original proof of Case 6 by Youngs \emph{et al.} had a few \emph{ad hoc} solutions and a general construction for $s \geq 4$. For $s \geq 2$, Youngs~\cite{Youngs-3569} gives a current graph construction for triangular embeddings of $K_{12s+6}-K_3$. The theory of current graphs is most suited for deleting a $K_3$ subgraph, but the Euler polyhedral equation does not rule out triangular embeddings of other graphs with the same number of edges and vertices. Gross~\cite{Gross-Case6} obtains triangular embeddings for some of these ``nearly complete'' graphs by modifying Youngs' constructions.

\begin{theorem}[Gross~\cite{Gross-Case6}]
For $s \geq 3$, there exists a triangular embedding of $K_{12s+6}-H$, where $H \in \{A, B, C, D, E\}$ is any of the five graphs on three edges in Figure~\ref{fig-bcde}. 
\label{thm-gross}
\end{theorem}

\begin{figure}[!ht]
\centering
\includegraphics{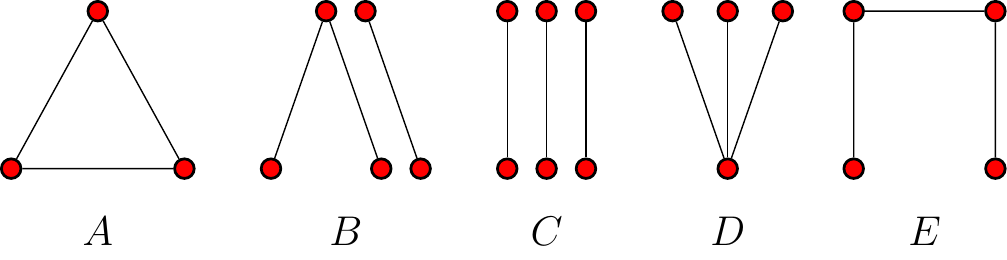}
\caption{The graphs on three edges.}
\label{fig-bcde}
\end{figure}

Before applying these embeddings for our task at hand, we extend this result one step further by filling in the case $s=2$. Youngs~\cite{Youngs-3569} also devised a current graph construction for $K_{30}-K_3$, which did not appear until Ringel's book~\cite{Ringel-MapColor}. We modify this embedding to get triangular embeddings of the other graphs.

\begin{proposition}
There exists a triangular embedding of $K_{30}-H$, where $H \in \{A, B, C, D, E\}$.
\label{prop-30} 
\end{proposition}
\begin{proof}

The current graph given by Ringel~\cite[p.155]{Ringel-MapColor} uses the group $\Z_{27}$ and produces the following three rows:

$$\begin{array}{rrrrrrrrrrrrrrrrrrrrrrrrrrrrrrrrrrrrrrrrrrrrr}
0. & 26 & 15 & 16 & 24 & 8 & 6 & 25 & 4 & 7 & 22 & 9 & 18 & 13 & z & \dots & \\
   & 14 & 1 & 12 & 11 & 3 & 19 & 21 & 2 & 23 & 20 & 5 & x & 10 & y & 17 & \\
   \\
1. & 0 & 14 & 10 & 19 & 5 & y & 18 & x & 26 & z & 15 & 6 & 13 & 23 & \dots & \\
   & 11 & 16 & 21 & 25 & 17 & 20 & 4 & 24 & 22 & 2 & 8 & 7 & 3 & 9 & 12 & \\
   \\
2. & 3 & 20 & 11 & 12 & y & 25 & 16 & 15 & z & 4 & x & 24 & 17 & 7 & \dots & \\
   & 19 & 14 & 9 & 5 & 13 & 10 & 26 & 6 & 8 & 1 & 22 & 23 & 0 & 21 & 18 & \\
\end{array}$$

We use a modified version of the additive rule to determine the remaining numbered rows---for row $k$, we take row $(k \bmod{3})$ and add $k-(k \bmod{3})$ to all the numbered entries. After manufacturing rows $x$, $y$, and $z$, we have a triangular embedding of $K_{30}-K_3$. 

When row $a$ is of the form $\dots c \,\,\, b \,\,\, d \dots$ and $(c,d)$ is not an edge in the graph, Gross~\cite{Gross-Case6} uses the notation $-(a,b)+(c,d)$ to denote an \emph{edge flip}, where we delete the edge $(a,b)$ and add the edge $(c,d)$ in the resulting quadrilateral. One can check that after applying the following groups of edge flips, we realize triangular embeddings of the four other graphs:

\begin{itemize}
\item $K_{30}-B$: ${-}(0,10){+}(x,y)$
\item $K_{30}-C$: ${-}(0,10){+}(x,y)$, ${-}(1,26){+}(x,z)$
\item $K_{30}-D$: ${-}(0,10){+}(x,y)$, ${-}(8,10){+}(x,z)$, ${-}(10,x){+}(y,z)$
\item $K_{30}-E$: ${-}(1,26){+}(x,z)$, ${-}(11,16){+}(1,26)$, ${-}(6,x){+}(11,16)$
\end{itemize}

\end{proof}

The graph we focus on particular is $K_n-E$, where $E = P_3$ is the path graph on three edges. Carefully adding these edges back yields a nearly triangular embedding.

\begin{lemma}
If there exists a triangular embedding of $K_n-P_3$, there exists an embedding of type $(6)$ of $K_n$.
\label{lem-addpath}
\end{lemma}
\begin{proof}
Suppose the missing edges are $(a,b)$, $(b,c)$, and $(c,d)$. The edges $(a,c)$ and $(b,d)$ are in the graph, so there are triangular faces $[a, c, x]$ and $[d, b, x']$ for some (possibly nondistinct) vertices $x$ and $x'$. With one handle, we can add back the missing edges following Figure~\ref{fig-knp3}, leaving the 6-sided face $[a, b, x', d, c, x]$.
\end{proof}

\begin{figure}[!ht]
\centering
\includegraphics{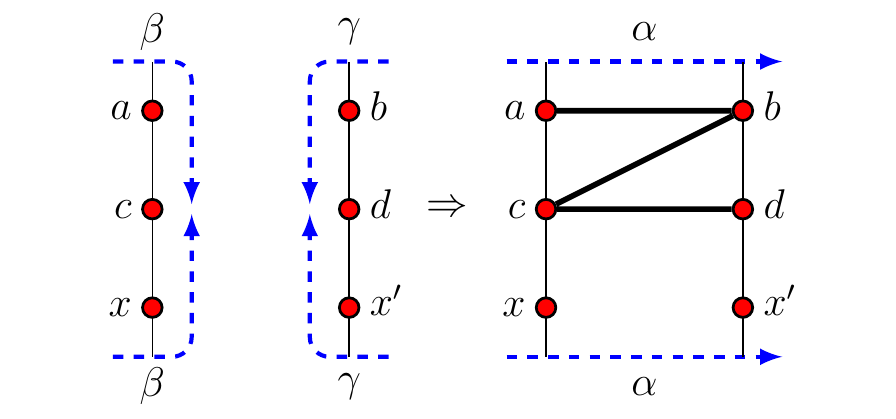}
\caption{Adding a $P_3$ subgraph using a specific pair of faces to get a 6-sided face.}
\label{fig-knp3}
\end{figure}

\begin{remark}
The approach of flipping edges in a triangulation to get the graph $K_n-P_3$ seems better suited for index 3 current graphs (see Youngs~\cite{Youngs-3569}), where the vortices can be nearly adjacent to each other in the log of the face boundary. The known current graph constructions for Cases 1 and 10 enjoy no such benefit.
\end{remark}

\section{Connecting three vertices with a handle}

So far, we have seen some examples of how to use a handle to add a few extra edges. In those cases, either two vertices or two pairs of adjacent vertices are joined together without disrupting any part of the rest of the embedding. Now, we show how to join three nonadjacent vertices using one handle. This construction was used in all additional adjacency steps in the original proof of the Map Color Theorem except Case 5. While Ringel and Youngs~\cite{RingelYoungs} illustrated this construction with drawings of the dual formulation of the problem, we elect to work in the primal to emphasize the  nontriangular faces and their incident vertices. 

\begin{construction}[On input vertices $v$; $x$, $y$, $z$]
Suppose the rotation at $v$ has the form
$$\begin{array}{rrrrrrrrrrrrrrrrrrrrrrrrrrrrrrrrrrrrrrrrrrrrr}
v. & x & a_1 \dots a_i & y & b_1 \dots b_j & z & c_1 \dots c_k.
\end{array}$$
Do the following:
\begin{itemize}
\item delete the edges $(v,x)$, $(v,y)$, and $(v,z)$ and
\item rewrite the rotation at $v$ as
$$\begin{array}{rrrrrrrrrrrrrrrrrrrrrrrrrrrrrrrrrrrrrrrrrrrrr}
v. & a_1 \dots a_i & c_1 \dots c_k & b_1 \dots b_j.
\end{array}$$
\end{itemize}
\label{construction-k3}
\end{construction}

\begin{proposition}
Applying Construction~\ref{construction-k3} on a triangular embedding of genus $g$ yields an embedding of genus $g+1$ with the 12-sided face
 $$[a_1 , x , c_k , v , b_1 , y , a_i , v , c_1 , z , b_j , v].$$
\end{proposition}
\begin{proof}
The 12-sided face is traced out in Figure~\ref{fig-heffter} using the Heffter-Edmonds principle. In fact, most of the faces remain intact, except those incident with $v$ and one of $x$, $y$, or $z$. The number of edges and faces decreased by 3 and 5, respectively, so by the Euler polyhedral equation, the genus increased by 1.
\end{proof}

\begin{figure}[!ht]
\centering
\includegraphics{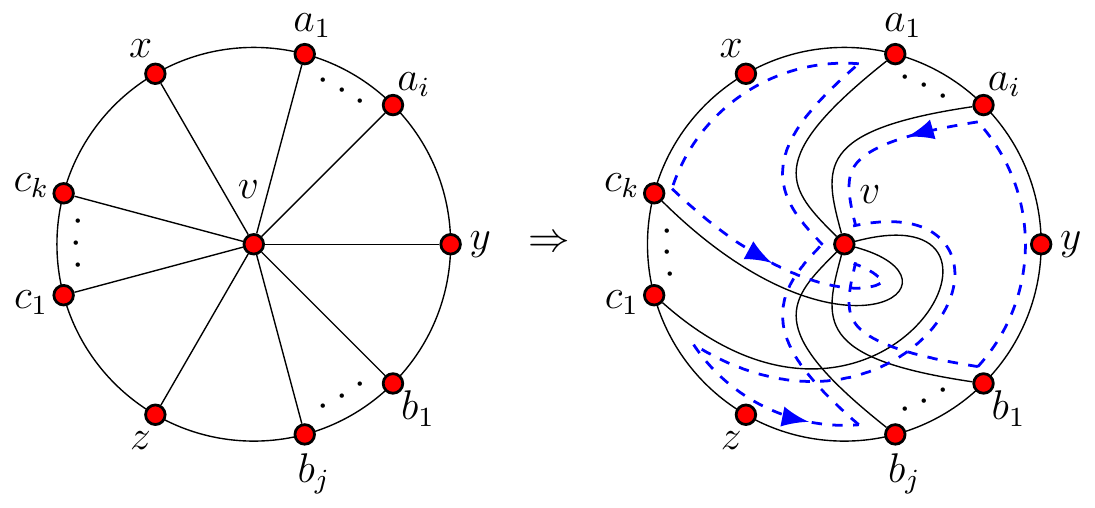}
\caption{Altering the rotation at a vertex and incrementing the genus. The dashed line indicates the boundary of the 12-sided face.}
\label{fig-heffter}
\end{figure}

Using this large face, we try to add back the edges we deleted and some others. The usual interpretation of Ringel's additional adjacency operation in the primal formulation is that three faces containing the vertices $x$, $y$, and $z$ are merged together, but note that for each of those vertices, we have two faces incident with that vertex and $v$. By first deleting the three edges $(v,x)$, $(v,y)$ and $(v,z)$, we can consider all possible combinations of faces simultaneously.

The most immediate application of this construction simply connects $x$, $y$, and $z$:

\begin{proposition}
If there exists a triangular embedding of $K_n-K_3$ then there exists an embedding of $K_n$ in the surface of genus $I(n)$. 
\end{proposition}
\begin{proof}
After adding the chords $(x,y)$, $(y,z)$, and $(x,z)$, we are left with the 5-sided faces $[0, b_1, y, x, c_k]$, $[0, c_1, z, y, a_i]$, and $[0, a_1, x, z, b_j]$, as in Figure~\ref{fig-555}. There are several options for adding back the edges $(0, x)$, $(0, y)$, and $(0, z)$ as chords. 
\end{proof}

\begin{figure}[!ht]
\centering
\includegraphics{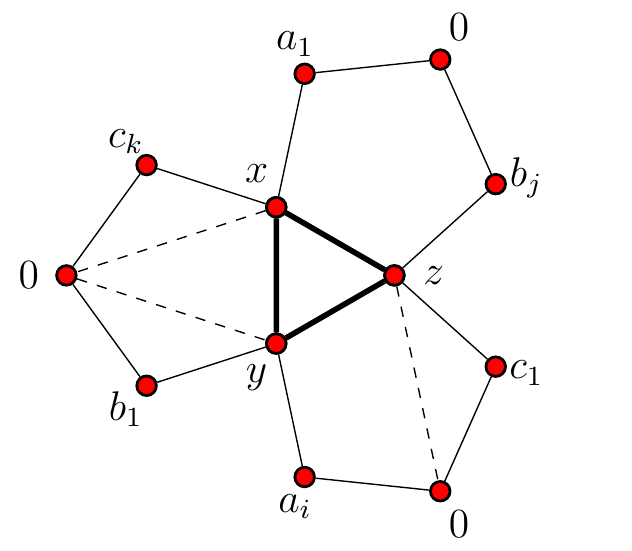}
\caption{Adding in a $K_3$ with a handle after deleting some edges. One possible way of restoring the deleted edges is shown with dashed lines.}
\label{fig-555}
\end{figure}

\begin{remark}
A triangular embedding of $K_n-K_3$ is known for $n \equiv 1, 6, 9, 10 \pmod{12}$, $n \geq 10$, so the correct choices of chords gives us embeddings of type $(5,4)$ and $(4,4,4)$. However, in light of Proposition~\ref{prop-simple6}, we do not need this result.
\end{remark}

\section{Case 10}

\begin{theorem}
For $s\geq 0$, there exists a nearly triangular minimum genus embedding of $K_{12s+10}$. 
\end{theorem}
\begin{proof}
For $s=0$, we apply Lemma~\ref{lem-addpath} to the triangular embedding of $K_{10}-P_3$ given in Table~\ref{tab-k10} in the Appendix. A unified solution is given for $s \geq 1$ in Theorem~\ref{thm-gen10}. 
\end{proof}
\begin{corollary}
For $s\geq 0$, there exist embeddings of type $(6)$, $(5,4)$, and $(4,4,4)$ of $K_{12s+10}$. 
\end{corollary}

For $s \geq 0$, Ringel~\cite[\S2.3]{Ringel-MapColor} gives a valid current graph generating $K_{12s+10}-K_3$, the $s=2$ case being illustrated in Figure~\ref{fig-cur-case10-graceful}. Luckily for us, the current assignments, which follow the same alternating pattern in the rungs of the ladder in Figure~\ref{fig-cur-case10-graceful}, can be used to produce a nearly triangular embedding.

\begin{figure}[!ht]
\centering
\includegraphics{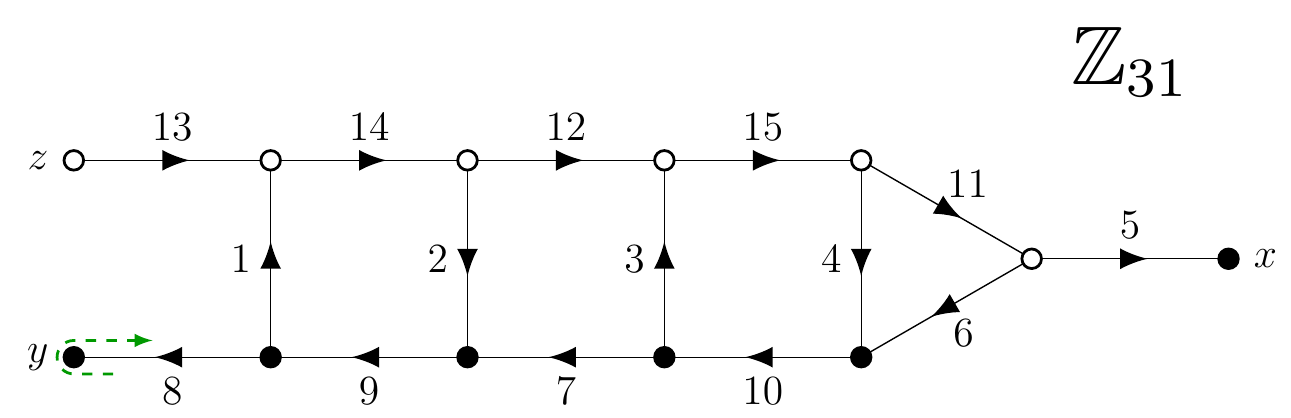}
\caption{The current graph for $s=2$, which produces a triangular embedding of $K_{34}-K_3$. }
\label{fig-cur-case10-graceful}
\end{figure}

\begin{theorem}
For $s \geq 1$, there exists a nearly triangular minimum genus embedding of $K_{12s+10}$.
\label{thm-gen10}
\end{theorem}
\begin{proof}

We use the same current graph as Ringel~\cite{Ringel-MapColor}, except we flip the rotation at the vertex adjacent to vortex $z$, as shown in Figure~\ref{fig-gen10}. Our solution to the additional adjacency problem hinges on the fact that the current $2s{+}1$ flowing into vortex $x$ is twice that of $-(5s{+}3)$, the current flowing into vortex $z$. Let $c = -(5s{+}3) = 7s{+}4$. Then $2s{+}1 = 2c$ and $3s{+}2 = -3c$ in the group $\Z_{12s+7}$. 

\begin{figure}[!ht]
\centering
\includegraphics{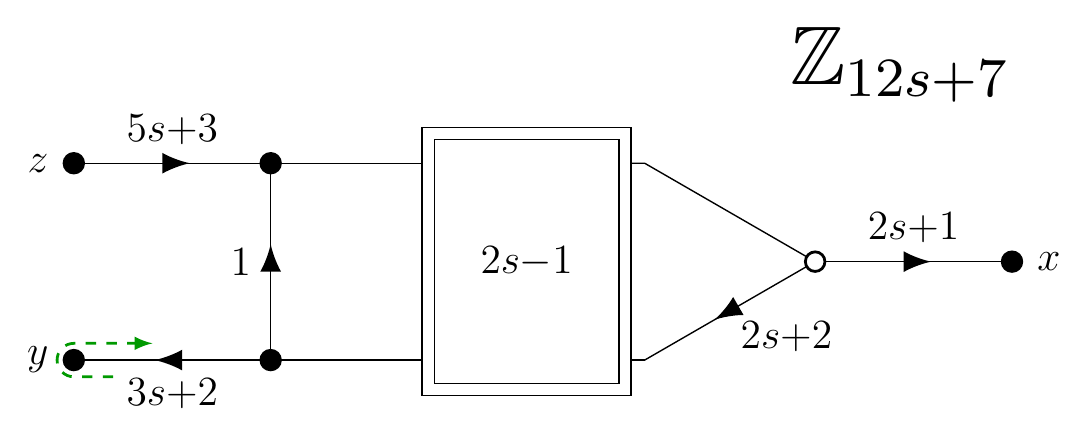}
\caption{A valid current graph generating $K_{12s+10}-K_3$ with the pertinent currents marked. }
\label{fig-gen10}
\end{figure}

After rewriting the currents in Figure~\ref{fig-gen10}, the log of this current graph and some partial rows become 
$$\arraycolsep=5pt\begin{array}{rlllllllllllll}
0. & {-}3c & y & 3c & 1 & c & z & {-}c & \dots & {-}2c{-}1 & 2c & x & {-}2c & \dots \\
c{+}1. & & & & & & & & \dots & {-}c & 3c{+}1 & x & \dots \\
2c. & & & \dots & 2c{+}1 & 3c & z & \dots\\
2c{+}1. & & & & \dots & 3c{+}1 & z & c{+}1 & \dots\\
\end{array}$$
In addition, row $x$ reads{
\setlength{\arraycolsep}{5pt}
$$\begin{array}{rrrrrrrrrrrrrrrrrrrrrrrrrrrrrrrr}
x. & \dots & {-}c & c & 3c & 5c & \dots
\end{array}$$
}
After applying Construction~\ref{construction-k3} to vertices $0$ and $x,y,z$, we obtain the 12-sided face
$$[x, 2c, 0, 3c, y, {-}3c, 0, {-}c, z, c, 0, {-}2c]$$
as in Figure~\ref{fig-c10s2}. We can exchange the chords $(x, c)$ and $(x,3c)$, generating the $5$-sided face $[x, {-}c, c, 3c, 5c]$. There remains only one way of adding back the edges $(0,y)$ and $(0,z)$. With the two remaining quadrilateral faces, we add $(0,x)$ to $[0, {-}2c, x, c]$, and on the other face, we start a sequence of chord exchanges
$$\pm (2c,3c) \pm (2c{+}1, z) \pm (c{+}1, 3c{+}1) \pm ({-}c, x).$$
These swaps are depicted in Figure~\ref{fig-c10s1}. Since the last edge was incident with the $5$-sided face, we get a nearly triangular embedding of $K_{12s+10}$. 
\end{proof}

\begin{figure}[!ht]
\centering
\includegraphics{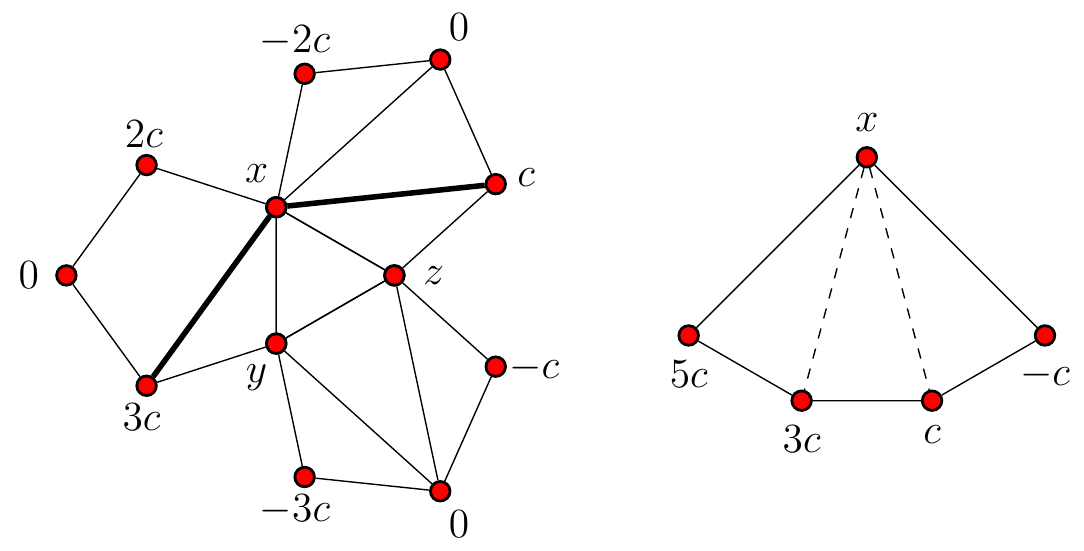}
\caption{Obtaining another 5-sided face using some simple chord exchanges. }
\label{fig-c10s2}
\end{figure}

\begin{remark}
There is a rich family of current assignments derived from graceful labelings of paths (see, e.g., Goddyn \emph{et al.}~\cite{Goddyn-Exponential}). Prior to discovering the solution presented here, the author found a more complicated family of graceful labelings for a similar additional adjacency solution. 
\end{remark}

\begin{figure}[!ht]
\centering
\includegraphics{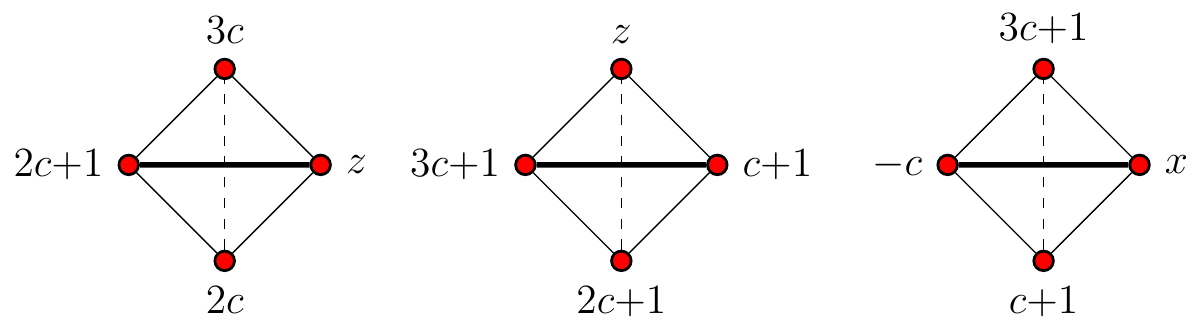}
\caption{Exchanging chords to get a 6-sided face.}
\label{fig-c10s1}
\end{figure}

\section{Case 1}\label{sec-c1}

\begin{theorem}
For $s \geq 1$, there exists a nearly triangular minimum genus embedding of $K_{12s+1}$. 
\end{theorem}
\begin{proof}
The minimum genus embedding of $K_{13}$ given by Ringel~\cite[p.82]{Ringel-MapColor} already happens to be nearly triangular. The remaining cases are handled by Theorem~\ref{thm-case1new} using current graphs.
\end{proof}
\begin{corollary}
For $s\geq 1$, there exist embeddings of type $(6)$, $(5,4)$, and $(4,4,4)$ of $K_{12s+1}$. 
\end{corollary}

Gustin (see Ringel~\cite[\S6.3]{Ringel-MapColor}) found the first complete solution for triangular embeddings of $K_{12s+1}-K_3$. Those current graphs are most elegantly described using the group $\Z_2 \times \Z_{6s-1}$, but since our general solution does not make use of this representation, we have relabeled Gustin's current graph for $s = 2$, as shown in Figure~\ref{fig-c1s2plain}.

\begin{figure}[!ht]
\centering
\includegraphics{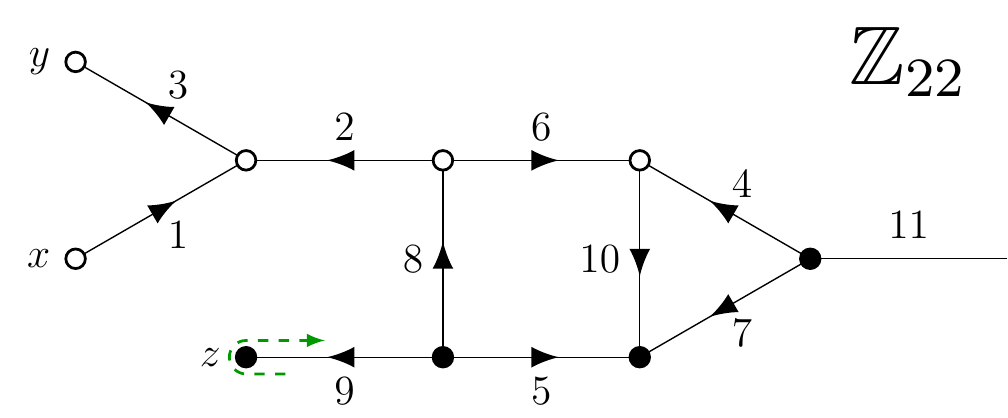}
\caption{Gustin's current graph relabeled.}
\label{fig-c1s2plain}
\end{figure}

\begin{theorem}
For $s \geq 2$, there exists a nearly triangular minimum genus embedding of $K_{12s+1}$.
\label{thm-case1new}
\end{theorem}

\begin{proof}
In addition to the current graph in Figure~\ref{fig-c1s2plain}, we also make use of Figure~\ref{fig-c1s3plain}, which gives a new triangular embedding of $K_{12s+1}-K_3$ for all $s \geq 3.$
The elements $1$, $3$, and $6s{-}3$ are all generators of $\Z_{12s-2}$, so the vortices are all of type (T1). We note that when $s = 3$, the ladder portion has exactly one rung labeled $9 = 6s{-}9$. 

\begin{figure}[!ht]
    \begin{subfigure}[b]{\textwidth}
        \centering
        \includegraphics[scale=0.9]{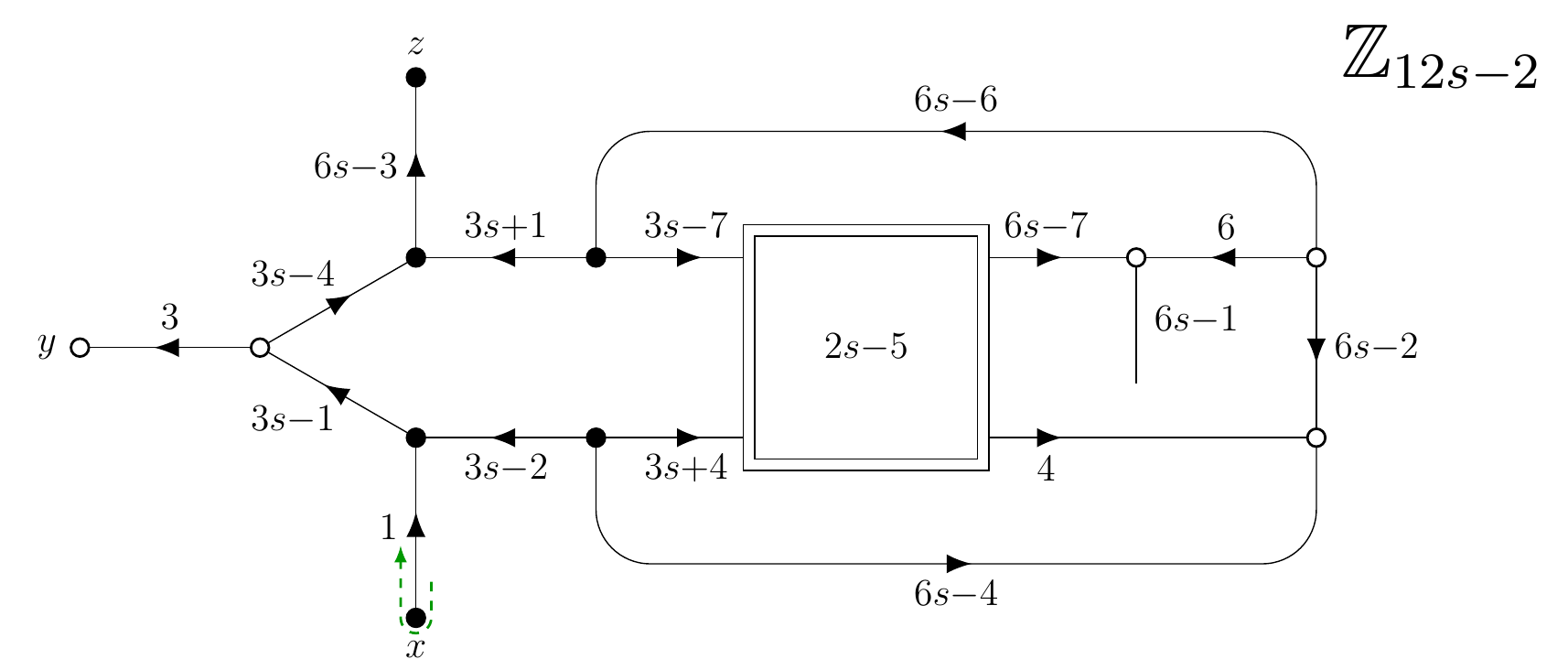}
        \caption{}
        \label{subfig-c1-a}
    \end{subfigure}
    \begin{subfigure}[b]{\textwidth}
        \centering
        \includegraphics[scale=0.8]{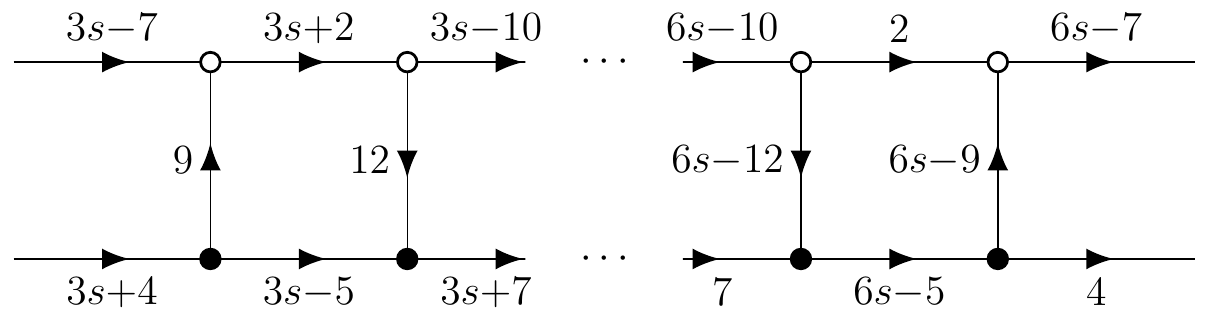}
        \caption{}
        \label{subfig-c1-b}
    \end{subfigure}
\caption{Current graphs for Case 1, $s \geq 4$. The box in the upper half (a) is replaced by the ladder in the bottom half (b).}
\label{fig-c1s3plain}
\end{figure}

For $s = 2$, the embedding produced from the current graph in Figure~\ref{fig-c1s2plain} is of the form 
$$\begin{array}{rcccccccccccccc}
0. & 17 & 9 & z & 13 & \dots & 3 & y & 19 & 21 & x & 1 & 20 & 14 & \dots \\
3. & & & & & & & & \dots & 2 & x & 4 & \dots \\
4. & & & & & & & & & & \dots & 5 & 2 & 18 & \dots \\
18. & 13 & 5 & z &  \dots\\
\end{array}$$
In the general case, we are interested in the following parts:
$$\begin{array}{rcccccccccccccccc}
0. & 6s{-}3 & z & 6s{+}1 & \dots & 6s{+}4 & 6 & 6s{+}5 & \dots & 3 & y & {-}3 & \dots & {-}1 & x & 1 & \dots \\
6s{-}3. & & & & & & & &  \dots & 6s & y & 6s{-}6 & \dots \\
6s{-}6. & & & & \dots & 0 & 6s & 1 & \dots \\
\end{array}$$
In both cases, the relative positions of the letters $x$, $y$, and $z$ in the rotation of 0 is the same, so applying Construction~\ref{construction-k3} on vertex $0$ and vertices $x,y,z$, we get the 12-sided face
$$[z, 6s{-}3, 0, {-}3, y, 3, 0, 1, x, {-}1, 0, 6s{+}1].$$
The sequences of chord exchanges
$$\pm(x,3)\pm(2,4)\pm(5, 18)\pm(z, 13)$$
for $s = 2$ and
$$\pm(y,6s{-}3)\pm(6s{-}6,6s)\pm(0, 1)$$
for $s \geq 3$ removes one of the edges incident with the 12-sided face. If we add the remaining edges according to Figure~\ref{fig-case1add}, we are left with a 6-sided face, indicating that the resulting embeddings are nearly triangular. 
\end{proof}

\begin{figure}[!ht]
\centering
    \begin{subfigure}[b]{0.4\textwidth}
        \centering
        \includegraphics[scale=1]{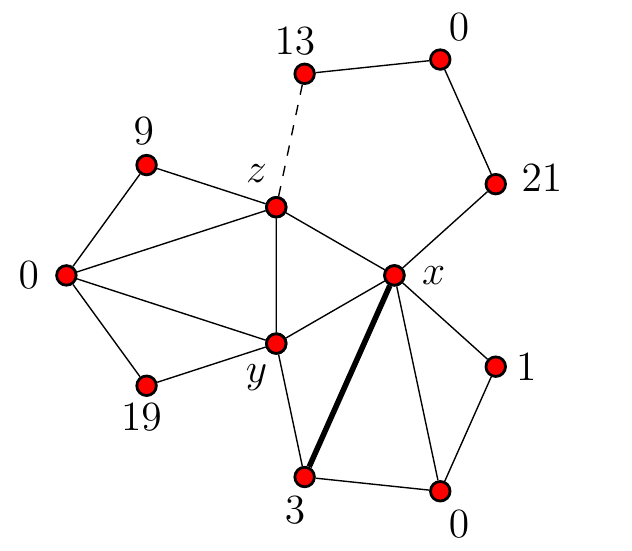}
        \caption{}
        \label{subfig-add1-a}
    \end{subfigure}
    \begin{subfigure}[b]{0.4\textwidth}
        \centering
        \includegraphics[scale=1]{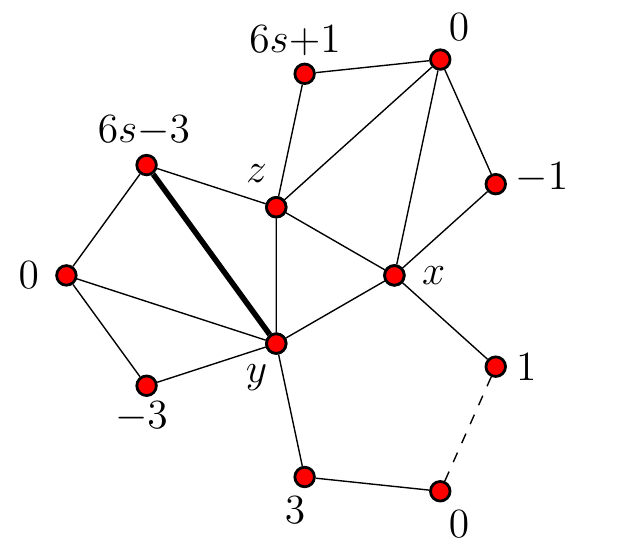}
        \caption{}
        \label{subfig-add1-b}
    \end{subfigure}
\caption{The end result of the Case 1 additional adjacency for $s = 2$ (a) and $s \geq 3$ (b).}
\label{fig-case1add}
\end{figure}

\begin{remark}
To the best of our knowledge, all previously published families of current graphs for orientable triangulations of $K_{12s+1}-K_3$ split into two subfamilies depending on the parity of $s$. Our current graphs in Figure~\ref{fig-c1s3plain} form a solution which handles all $s \geq 3$ irrespective of parity. Another such family of current graphs is presented in Appendix~\ref{app-c1} that extends to the $s=2$ case, though we were unable to use it to prove Theorem~\ref{thm-case1new}.
\end{remark}

\section{Case 8}\label{sec-case8}

\begin{theorem}
For $s\geq 1$, there exists a nearly triangular minimum genus embedding of $K_{12s+8}$. For $s = 0$, there does not exist such an embedding. 
\end{theorem}
\begin{proof}
Corollary~\ref{corollary-k8} shows nonexistence for $K_8$. All remaining values of $s$ are covered by Theorem~\ref{thm-gen8} with a unified additional adjacency step.
\end{proof}
\begin{corollary}
For $s\geq 1$, there exist embeddings of type $(5)$ and $(4,4)$ of $K_{12s+8}$. All minimum genus embeddings of $K_8$ are of type $(4,4)$. 
\end{corollary}
\begin{proof}
For the exceptional case $s = 0$, the graph $G_9$ in Theorem~\ref{thm-case9graphs} (see Ringel~\cite[p.79]{Ringel-MapColor}) has two vertices $x_0$ and $x_1$ of degree 4. Deleting both those vertices leaves an embedding of $K_8$ of type $(4,4)$. 
\end{proof}

We first show the nonexistence of a nearly triangular embedding of $K_8$, which was also verified by an exhaustive computer search. The proof relies on another nonexistence result for so-called ``minimum triangulations'' of surfaces.

\begin{theorem}[Huneke~\cite{Huneke-Minimum}]
If a simple graph $G$ triangulates $S_2$, then $G$ must have at least 10 vertices.
\label{ref-huneke}
\end{theorem}

\begin{corollary}
$K_8$ does not have a nearly triangular minimum genus embedding.
\label{corollary-k8}
\end{corollary}
\begin{proof}
Suppose such an embedding exists. The Map Color Theorem states that the minimum genus of $K_8$ is 2, and furthermore, a nearly triangular embedding in the surface $S_2$ would have a simple 5-sided face as a consequence of Proposition~\ref{prop-trianglebound} and Lemma~\ref{lem-5simp}. Subdivide the face by adding a new vertex $v$ inside the face and add edges to connect $v$ to the vertices on the boundary of the face. Now, we have a triangular embedding of the simple graph $K_9-D$ in $S_2$, where $D = K_{1,3}$ is shown in Figure~\ref{fig-bcde}. However, Theorem~\ref{ref-huneke} states that no such embedding exists. 
\end{proof}

In the additional adjacency steps of both Cases 8 and 11, edge flips are used to sacrifice one existing edge to gain a previously missing edge. For example, suppose we had the following partial table of a triangular embedding:
$$\arraycolsep=5pt\begin{array}{rlllllllllllll}
a. & \dots & c & b & d & \dots \\
c. & \dots & e & d & f & \dots \\
e. & \dots & g & f & h & \dots \\
\end{array}$$
and $(g,h)$ is not an edge of the graph. Then we can do the edge flips
\begin{align*}
-(e,f)&+(g,h)\\
-(c,d)&+(e,f)\\
-(a,b)&+(c,d)
\end{align*}
to add $(g,h)$ at the cost of $(a,b)$. For brevity, we write this operation as the \emph{sequence of edge flips}
$$-(a,b)\pm(c,d)\pm(e,f)+(g,h).$$
The notation suggests that we can view this operation alternatively as deleting the edge $(a,b)$, exchanging the chords $(c,d)$ and $(e,f)$, and then finally adding $(g,h)$. 

\begin{theorem}
For $s \geq 1$, there exists a nearly triangular minimum genus embedding of $K_{12s+8}$. 
\label{thm-gen8}
\end{theorem} 
\begin{proof}
We will use the novel family of current graphs in Figure~\ref{fig-newgraph} for all $s \geq 3$. 
For $s = 2$, we appeal to Ringel and Youngs~\cite{RingelYoungs-Case8} for the valid current graph in Figure~\ref{fig-cur-case8}, and for $s=1$, we use the rotation system generated by the rows
$$\begin{array}{rrrrrrrrrrrrrrrrrrrrrrrrrrrrrrrrrrrrrrrrrrrrr}
0. & 1 & 5 & 3 & 14 & 11 & 15 & 2 & 7 & 12 & 16 & 4 & 6 & 13 & 10 & y & 8 & 9 & 17 & x \\
1. & 0 & x & 2 & 16 & 6 & 12 & 14 & 7 & 3 & 15 & 13 & 8 & 10 & 17 & 11 & y & 9 & 4 & 5 \\
2. & 0 & 15 & 16 & 1 & x & 3 & 12 & y & 10 & 8 & 17 & 6 & 9 & 5 & 11 & 14 & 4 & 13 & 7
\end{array}$$
and the group $\Z_{18}$. The resulting triangulations have vertices $0, 1, \dotsc, 12s{+}5, x, y_0, y_1$, where all the numbered vertices are adjacent, $x$ is adjacent to all the numbered vertices, and $y_0$ and $y_1$ are adjacent to all the even and odd numbered vertices, respectively. We use one handle to connect $y_0$, $y_1$, and $x$. Then, contracting the edge $(y_0, y_1)$ yields a minimum genus embedding of $K_{12s+8}$. Initially, however, there is no vertex adjacent to all three lettered vertices.

\begin{figure}[!ht]
\centering
\includegraphics[scale=0.9]{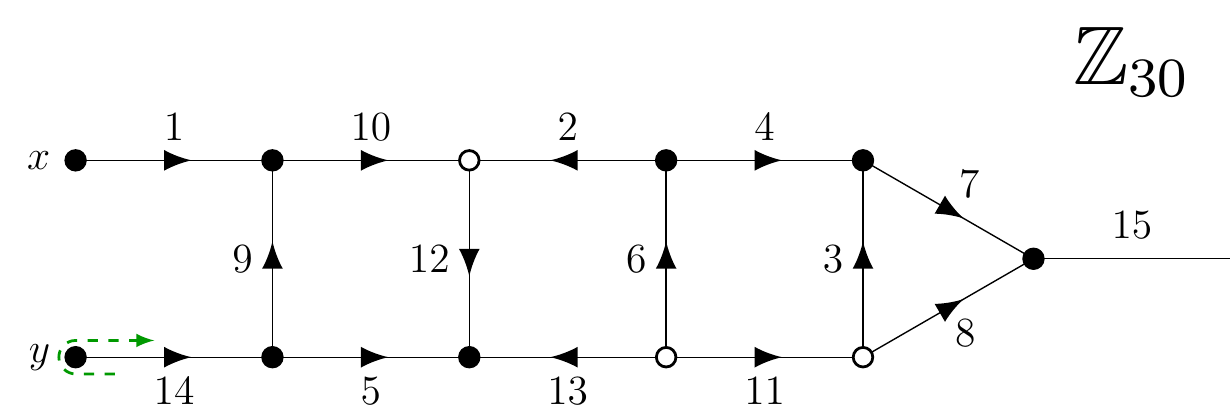}
\caption{The current graph of Ringel and Youngs~\cite{RingelYoungs-Case8} for $K_{32}$.}
\label{fig-cur-case8}
\end{figure}

\begin{figure}[!ht]
\centering
    \begin{subfigure}[b]{\textwidth}
        \centering
        \includegraphics[width=\textwidth]{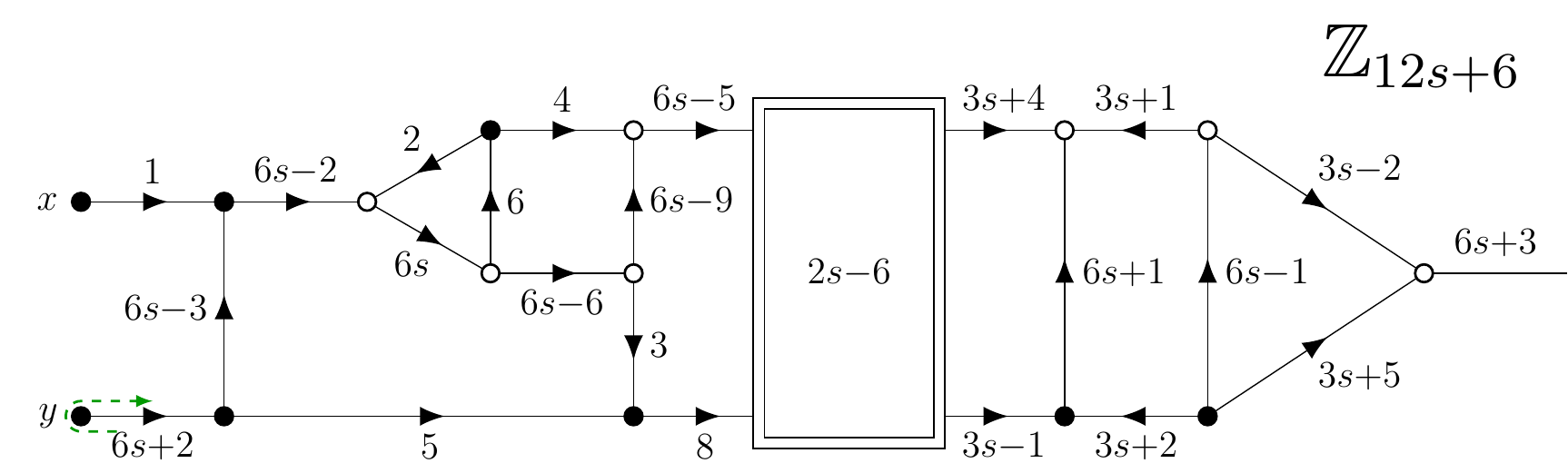}
        \caption{}
        \label{subfig-c8-a}
    \end{subfigure}
    \begin{subfigure}[b]{\textwidth}
        \centering
        \includegraphics[scale=0.8]{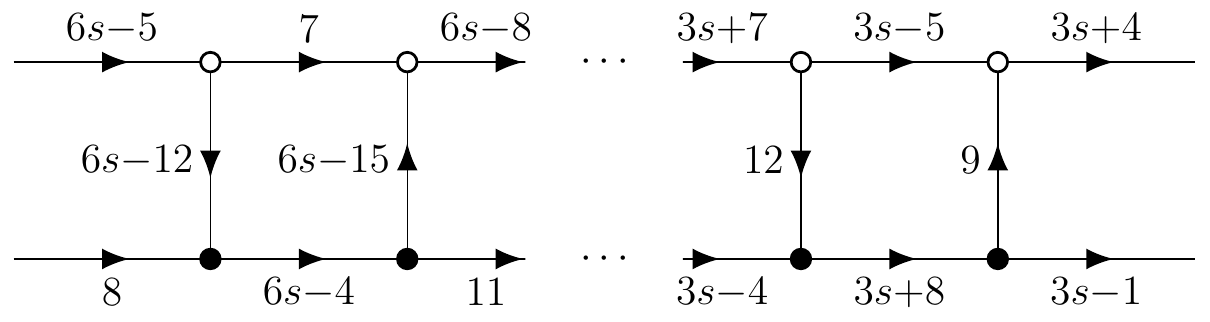}
        \caption{}
        \label{subfig-c8-b}
    \end{subfigure}
\caption{A new family of current graphs for $s\geq3$.}
\label{fig-newgraph}
\end{figure}

For $s \geq 2$, row $0$ of the embedding is of the form
{
\setlength{\arraycolsep}{5pt}
$$\begin{array}{rrrrrrrrrrrrrrrrrrrrrrrrrrrrrrrr}
0. & 6s{+}4 & y_0 & 6s{+}2 & \dots & {-}1 & x & 1 & \dots & 6s{+}9 & 5 & \dots & 6s{+}6 & 12s{+}4 & 4 & \dots
\end{array}$$
}
In all cases, including $s=1$, employing the additivity rule yields the following partial rows:
\begin{align*}
6s{-}1&. \,\,\,\, \dots \,\,\,\, 6s{-}2 \,\,\,\, x \,\,\,\, 6s \,\,\,\, \dots \\
6s&. \,\,\,\, \dots \,\,\,\, 0 \,\,\,\, 6s{-}2 \,\,\,\, 6s{+}4 \,\,\,\, \dots 
\end{align*}
and a slightly more descriptive partial row for $12s{+}1$:
{
\setlength{\arraycolsep}{5pt}
$$\begin{array}{rrrrrrrrrrrrrrrrrrrrrrrrrrrrrrrr}
12s{+}1. & \dots & 6s{-}1 & y_1 & 6s{-}3 & \dots & 12s & x & 12s{+}2 & \dots & 6s{+}4 & 0 & \dots
\end{array}$$
}

Judging from these rows, we can perform the sequence of edge flips 
$$-(6s{-}1,x)\pm(6s,6s{-}2)\pm(0,6s{+}4)+(y_0,12s{+}1),$$ 
to produce a vertex adjacent to all three of $x$, $y_0$, and $y_1$ in preparation for Construction~\ref{construction-k3}. 
Row $12s{+}1$ now is of the form
{
\setlength{\arraycolsep}{5pt}
$$\begin{array}{rrrrrrrrrrrrrrrrrrrrrrrrrrrrrrrr}
12s{+}1. & \dots & 6s{-}1 & y_1 & 6s{-}3 & \dots & 12s & x & 12{+}2 & \dots & 6s{+} 4 & y_0 & 0 & \dots
\end{array}$$
}
as illustrated in Figure~\ref{fig-case8-initial}. If we apply Construction~\ref{construction-k3} to vertex $12s{+}1$ and neighbors $y_0, y_1, x$, we obtain the 12-sided face
$$[y_0, 6s{+}4, 12s{+}1, 6s{-}3, y_1, 6s{-}1, 12s{+}1, 12s{+}2, x, 12s, 12s{+}1, 0].$$
Adding the edge $(y_0, y_1)$ in this face and contracting it to make a new vertex $y$ yields one 4-sided face and one 8-sided face, and the remaining edges $(x,y)$, $(y, 12s{+}1)$, $(x, 12s{+}1)$, and $(x, 6s{-}1)$ can be added back in, pursuant to Figure~\ref{fig-case8-48}, to produce an embedding of type $(5)$. 
\end{proof}

\begin{figure}[!ht]
\centering
\includegraphics{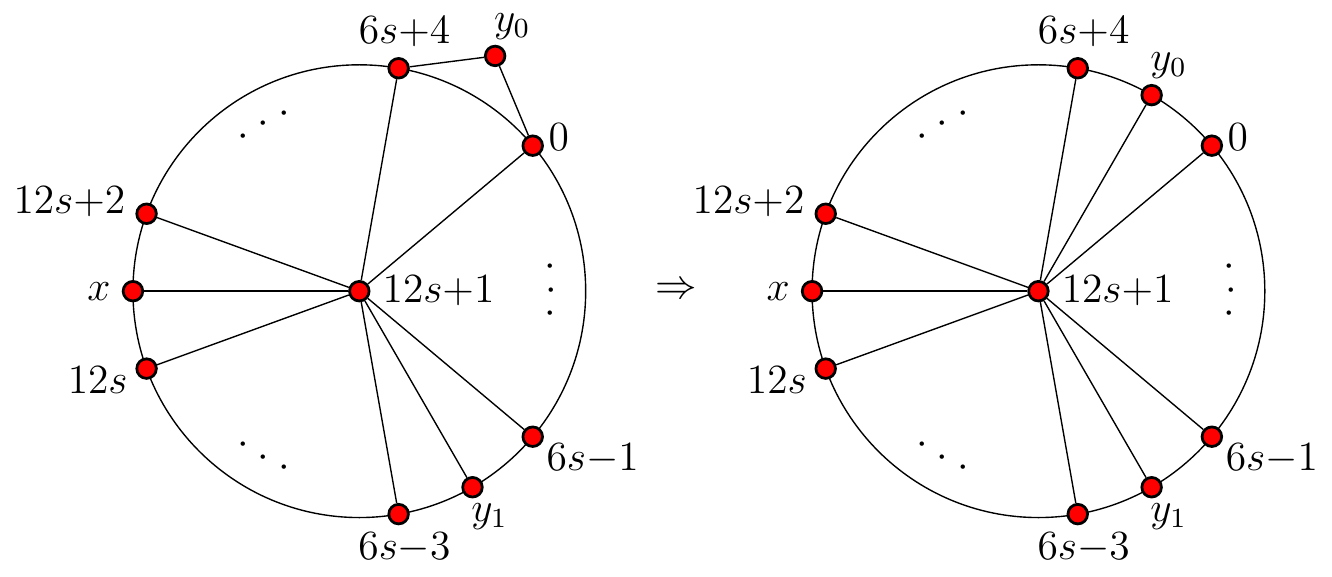}
\caption{The rotation at vertex $12s{+}1$ after the initial modifications. }
\label{fig-case8-initial}
\end{figure}

\begin{figure}[!ht]
\centering
\includegraphics{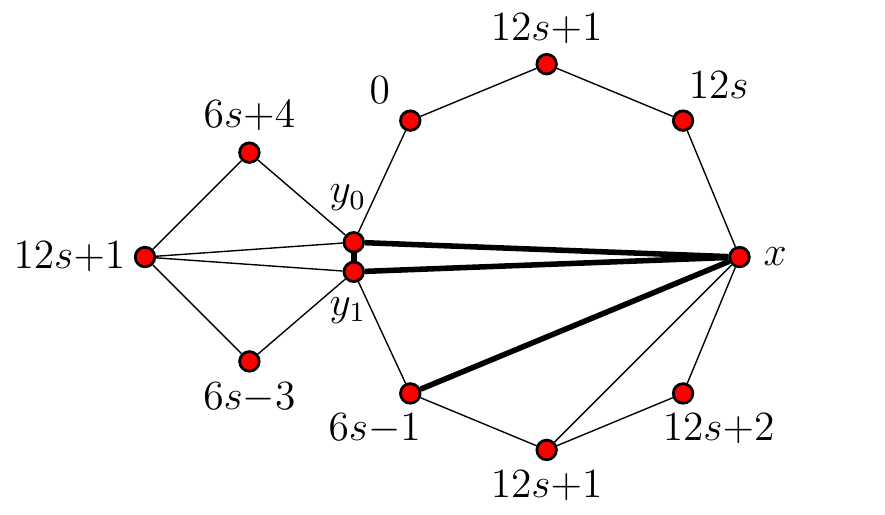}
\caption{Using one handle to connect $x$ with $y$ and to replace the missing edge $(x, 6s{-}1)$. The edge $(y_0, y_1)$ is contracted and the amalgamated vertex is renamed $y$. } 
\label{fig-case8-48}
\end{figure}

\begin{remark}
The rotation system used for $s=1$ can be interpreted as an index 3 current graph. To our knowledge, we have given the first minimum genus embedding of $K_{20}$ succinctly derived from a symmetric embedding.

We made use of a current graph of Ringel and Youngs~\cite{RingelYoungs-Case8}, but we did not include any of their other constructions. In fact, their family of current graphs for $s \geq 4$ are also applicable for the additional adjacency solution presented here. Our family of current graphs for $s \geq 3$, while slightly more complicated in terms of the underlying graph, benefits from a significantly simpler current assignment, where the generalization is, like Figure~\ref{fig-c1s3plain} for Case 1, a simple zigzag. This pattern is ``smooth'' in the sense of Guy and Youngs~\cite{GuyYoungs}. In addition, our solution handles the odd and even $s$ cases simultaneously, and it extends downwards to $s=3$, for which Ringel and Youngs~\cite{RingelYoungs-Case8} needed a special solution. 
\end{remark} 

\section{Case 11}

\begin{theorem}
For $s \geq 0$, there exists a nearly triangular minimum genus embedding of $K_{12s+11}$.
\end{theorem}
\begin{proof}
The embedding of $K_{11}$ given by Mayer~\cite{Mayer-Orientables}, after deleting two extra edges, is nearly triangular.\footnote{The embedding given in Ringel~\cite[p.81]{Ringel-MapColor}, results from deleting the ``wrong'' edge of each doubled pair, leaving an embedding of type $(4,4)$.} The embedding of $K_{23}$ we give in Table~\ref{tab-k23} was also found starting from Mayer~\cite{Mayer-Orientables} (see also Ringel~\cite[p.85]{Ringel-MapColor}). Two sequences of chord exchanges, starting with $(8,22)$ and $(10,16)$, eventually ``collide'' at two edges incident with the same face, resulting in a 5-sided face. 

The general case $s \geq 2$ is proved in Theorem~\ref{thm-gen11}.
\end{proof}
\begin{corollary}
For $s\geq 0$, there exist embeddings of type $(5)$ and $(4,4)$ of $K_{12s+11}$. 
\end{corollary}

\begin{theorem}
For $s \geq 2$, there exists a nearly triangular minimum genus embedding of $K_{12s+11}$.
\label{thm-gen11}
\end{theorem}
\begin{proof}

Ringel and Youngs~\cite{RingelYoungs-Case11} found current graphs with the geometry of Figure~\ref{fig-gen11} for $s \geq 2$. The current graphs produce triangular embeddings of $K_{12s+11}-K_5$, so the goal is to add in the edges between the lettered vertices using two handles. Near the vortices, the logs of both current graphs are
{
\setlength{\arraycolsep}{5pt}
$$\begin{array}{rrrrrrrrrrrrrrrrrrrrrrrrrrrrrrrr}
0. & x & 6s{+}5 & 12s{+}4 & a & 12s{+}5 & y & 1 & b & 12s{+}2 & \dots & 4 & c & 2 & \dots
\end{array}$$
}
\begin{figure}[!ht]
\centering
\includegraphics[scale=0.9]{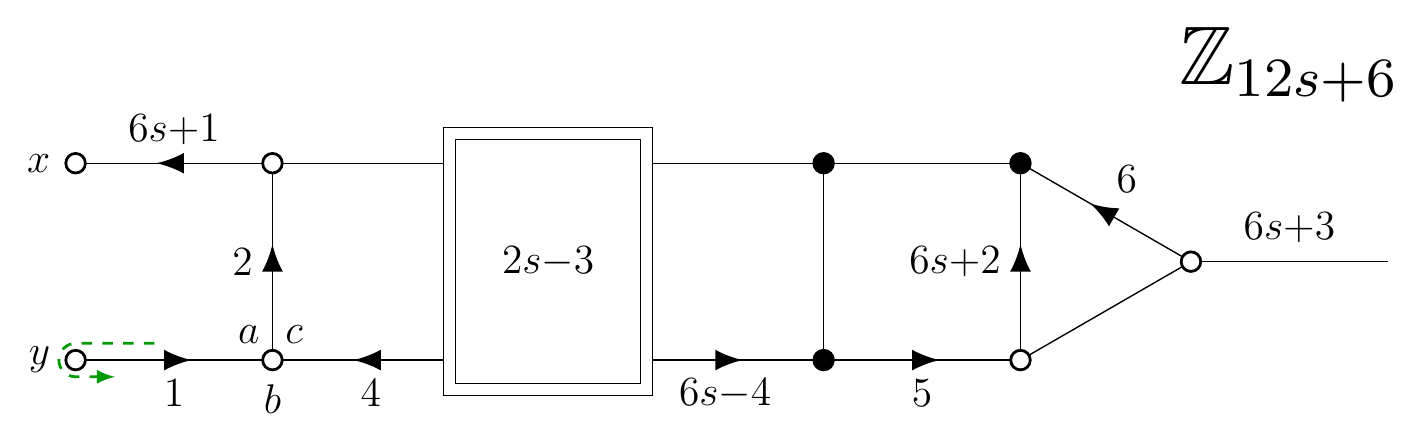}
\includegraphics[scale=0.9]{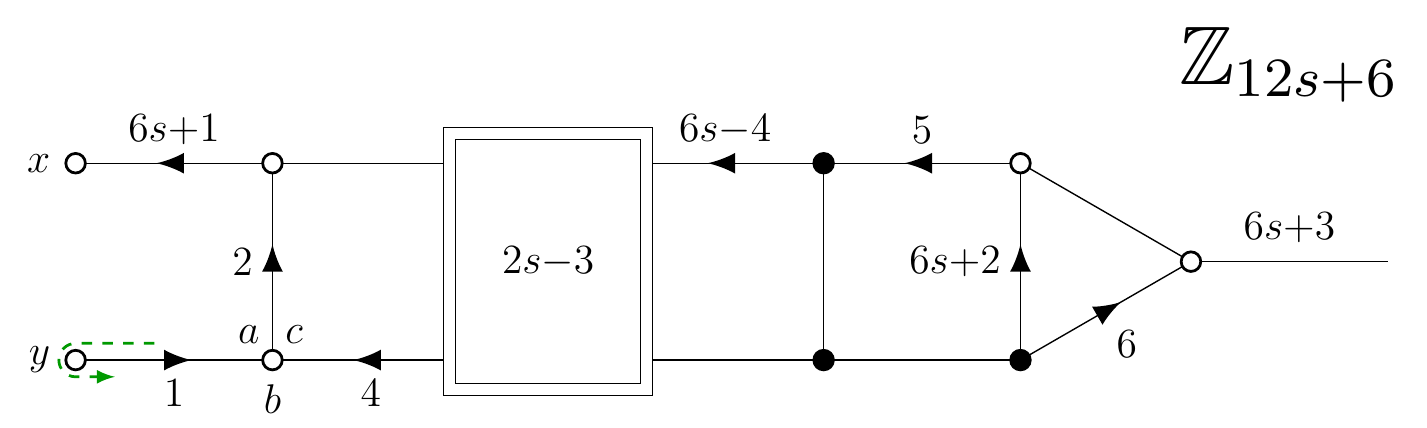}
\caption{The geometry of two general current graphs for Case 11, depending on the parity of $s$.}
\label{fig-gen11}
\end{figure}
Before adding handles, several local edge additions and deletions are made to the triangular embedding of $K_{12s+11}-K_5$. We omit the exact details of these modifications, which are identical to those in Ringel and Youngs~\cite{RingelYoungs-Case11} (see also Ringel~\cite[p.100]{Ringel-MapColor}). In summary, the resulting embedding now has the edges $(a,y)$, $(b,y)$, and $(a,x)$ at the expense of $(0, 12s{+}4)$, $(0, 6s{+}5)$, $(c,12s{+}4)$, and $(b,4)$. The embedding also has a single nontriangular face $[a, 12s{+}4, 6s{+}5, x]$, as seen in Figure~\ref{fig-mod11}. 

\begin{figure}[!ht]
\centering
\includegraphics{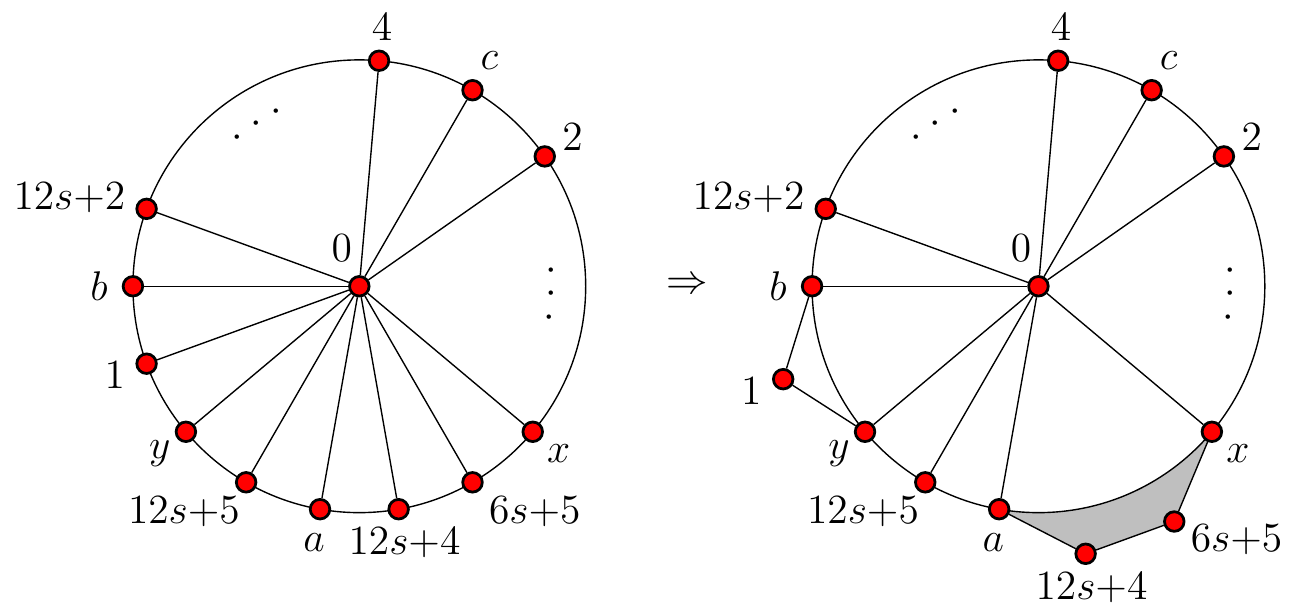}
\caption{Modifications to the rotation at vertex 0. The shaded quadrilateral face on the right will be used again later on.}
\label{fig-mod11}
\end{figure}

Applying Construction~\ref{construction-k3} to vertex $0$ and nonadjacent vertices $a, b, c$, we obtain the 12-sided face
$$[0, 12s{+}5, a, x, 0, 12s{+}2, b, y, 0, 2, c, 4]$$
while losing the edges $(0,a)$, $(0,b)$ and $(0,c)$. In this face, we add the chords $(0,a)$, $(0,b)$, $(0,c)$, $(a,b)$, $(b,c)$, $(c,y)$, $(b,4)$, $(b,x)$ as in Figure~\ref{fig-big11}(a). The handle creates the face $[0, c, y]$, and from the previous modifications, there is the quadrilateral $[x, a, 12s{+}4, 6s{+}5]$. Using another handle, we can merge the two faces to add the edges $(a, c)$, $(c, x)$, $(x, y)$, $(0, 6s{+}5)$, $(0, 12s{+}4)$, and $(c, 12s{+}4)$ as in Figure~\ref{fig-big11}(b). 

\begin{figure}[!ht]
    \centering
    \begin{subfigure}[b]{0.35\textwidth}
    \centering
        \includegraphics{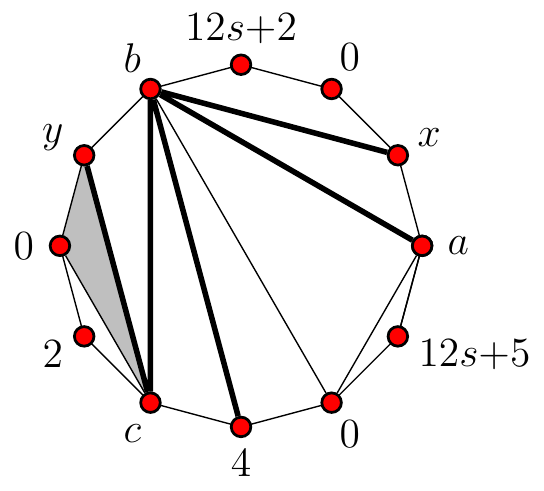}
        \caption{}
        \label{subfig-11a}
    \end{subfigure}
    \begin{subfigure}[b]{0.35\textwidth}
    \centering
        \includegraphics{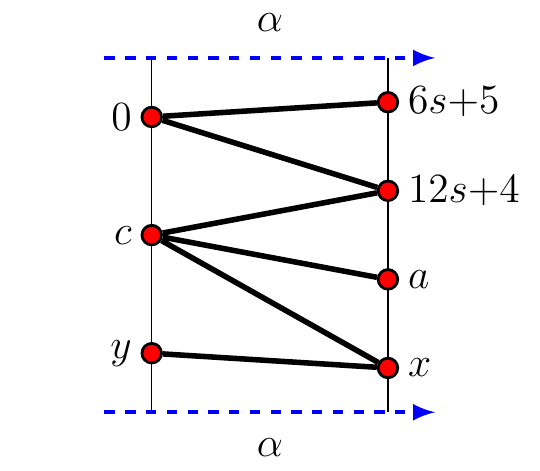}
        \caption{}
        \label{subfig-11b}
    \end{subfigure}
\caption{The gained edges from two handles. Note that the second handle in part (b) makes use of the shaded faces from Figure~\ref{fig-mod11} and part (a).}
\label{fig-big11}
\end{figure}

Now, all the missing edges have been added and we are left with an embedding of $K_{12s+11}$ with two quadrilateral faces
$$[0, 6s{+}5, x, y] \textrm{ and } [0, 12s{+}2, b, x].$$
Exchanging the chord $(0,x)$ yields an embedding of type $(5)$, completing the construction. 
\end{proof}

\section{Nonorientable embeddings}\label{sec-nonorient}

Let $N_k$ denote the nonorientable surface of genus $k$, a sphere with $k$ crosscaps, and let the \emph{minimum nonorientable genus} $\overline{\gamma}(G)$ be the genus of the smallest nonorientable surface that $G$ embeds in.\footnote{Formally, we include the sphere, which is orientable, as the nonorientable surface of genus $0$.} Analogously, we have the \emph{nonorientable} Euler polyhedral equation 

$$|V(G)| - |E(G)| + |F(G)| = 2-\overline{g}(N)$$

and the \emph{nonorientable} Map Color Theorem 

$$\overline{\gamma}(K_n) = \left\lceil \frac{(n-3)(n-4)}{6} \right\rceil, n \geq 3, n \neq 7.$$

The discrepancies with the orientable versions are due to the fact that one handle in a nonorientable surface is homeomorphic to two crosscaps. The lone exception $n = 7$ is due to Franklin~\cite{Franklin-SixColor}, who showed that $K_7$ cannot embed in $N_2$, the Klein bottle. 

Because crosscaps are ``half of a handle,'' we can obtain nonorientable triangular embeddings for some complete graphs that cannot triangulate an orientable surface. For $n \equiv 1, 6, 9, 10 \pmod{12}$, we showed that there were embeddings of type $(6)$ of $K_n$, but these graphs actually have nonorientable triangular embeddings (see Ringel~\cite{Ringel-MapColor}). Similarly, we used a handle to add the missing edge to a triangular embedding of $K_n - K_2$ for $n \equiv 2,5 \pmod{12}$, but actually a crosscap suffices. We summarize the expected types of embeddings we need to find:

\begin{itemize}
\item For $n \equiv 2, 5, 8, 11 \pmod{12}$, types $(5)$ and $(4,4)$.
\item For $n = 7$, types $(6)$, $(5, 4)$, and $(4,4)$. 
\end{itemize}

The situation for nonorientable genus embeddings is, like in the proof of the nonorientable Map Color Theorem, significantly simpler than its orientable counterpart. Instead of formally describing graph embeddings in nonorientable surfaces, we invoke the nonorientable version of Rule R* for the regular parts, and describe how to add crosscaps locally for the additional adjacency parts. The proof makes use of nonorientable current graphs, known as cascades, for constructing rotation systems on surfaces. We do not go into their definition because we will only need to focus on the additional adjacency part. Our contribution here is a careful observation of the embeddings produced by the original proof (see, e.g., Ringel~\cite{Ringel-MapColor}) in primal form.

\begin{definition}
A rotation system satisfies \emph{Rule R} if for all edges $(i,k)$, if row $i$ is of the form
$i. \, \dots \, j \, k \, l \dots,$
then row $k$ is either of the form $k. \, \dots l \, i \, j \dots$ or $k. \, \dots j \, i \, l \dots$
\end{definition}

\begin{theorem}[see Ringel~\cite{Ringel-MapColor}, Theorem 5.2]
If a rotation system of $G$ satisfies Rule $R$, then there exists a triangular embedding of $G$ on a (possibly nonorientable) surface.
\end{theorem}

Note the caveat in the above statement that the surface might be orientable. However, in the additional adjacency part, we add crosscaps, which always make the resulting surface nonorientable. To add a crosscap to a surface, we cut out a disk and identify opposite points of the resulting boundary. We provide nonorientable analogues of Construction~\ref{construction-k3}, using one and two crosscaps. 

\begin{proposition}
Suppose there exists a triangular embedding of $K_n-K_2$ in a (possibly nonorientable) surface. Then there exist nonorientable embeddings of $K_n$ of type $(5)$ and $(4,4)$.
\end{proposition}
\begin{proof}
Let the two nonadjacent vertices be $x$ and $y$, and let $0$ be a vertex adjacent to both. As seen in Figure~\ref{fig-crosscap-k2}, by deleting the edges $(0,x)$ and $(0,y)$ and passing some of the other edges incident with vertex $0$ through a crosscap, we obtain a 8-sided face incident with $x$, $y$, and two instances of $0$. After adding the chord $(x,y)$, there are a few choices of adding back the removed edges $(0,x)$ and $(0,y)$. Depending on the choice, we get an embedding of type $(5)$ or of type $(4,4)$. 

\begin{figure}
\centering
\includegraphics{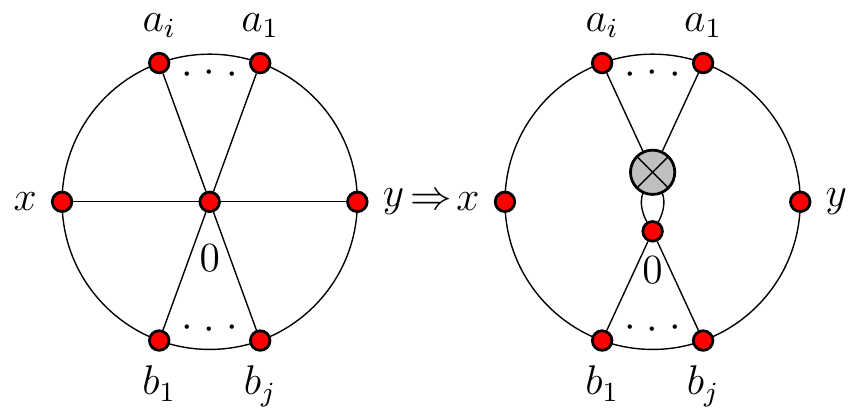}
\quad
\includegraphics{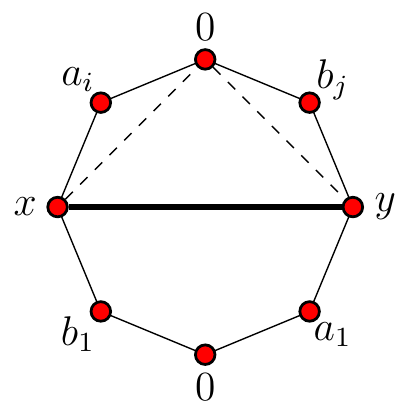}
\caption{One crosscap can be used to add an edge between nonadjacent vertices. The missing edges can be added back into the resulting $8$-sided face. In this example, the embedding is nearly triangular.}
\label{fig-crosscap-k2}
\end{figure}
\end{proof}

\begin{theorem}[see Ringel~{\cite[\S8.3]{Ringel-MapColor}}]
There exist triangular embeddings of $K_n-K_2$ for $n \equiv 5, 11\pmod{12}$.\footnote{Here, we could have also included the orientable embeddings of $K_{12s+2}-K_2$ referenced in Theorem~\ref{thm-case2} and the triangular embeddings of $K_{12s+8}-K_2$ of Korzhik~\cite{Korzhik-Case8} for $s \geq 1$.} 
\end{theorem}
\begin{corollary}
For $s \geq 0$, there exist nonorientable embeddings of type $(5)$ and $(4,4)$ of $K_{12s+5}$ and $K_{12s+11}$.
\end{corollary}

For Case 8, a current graph similar to Figure~\ref{fig-cur-case8} is used, where there is one vortex $x$ of type (T1), and one vortex $y$ of type (T2). The flexibility of nonorientability makes the additional adjacency problem significantly simpler.

\begin{theorem}
For $s \geq 0$, there exist nonorientable embeddings of type $(5)$ and $(4,4$) of $K_{12s+8}$.
\end{theorem}
\begin{proof}

The aforementioned current graph~\cite[Fig 8.27]{Ringel-MapColor} generates a triangular embedding of a graph $G_{12s+8}$ with vertices $0, 1, 2, \dotsc, 12s{+}5, x, y_0, y_1$ where all the numbered vertices are pairwise adjacent, $x$ is adjacent to all the numbered vertices, and $y_0$ (respectively, $y_1$) is adjacent to all the even- (respectively, odd-) numbered vertices. We use two crosscaps to connect $x$, $y_0$ and $y_1$. Since $x$ is incident with all the numbered vertices, the row of $x$ must be of the form
{
\setlength{\arraycolsep}{5pt}
$$\begin{array}{rrrrrrrrrrrrrrrrrrrrrrrrrrrrrrrr}
x. & \dots & \alpha & \beta & \dots
\end{array}$$
}
where $\alpha$ is even and $\beta$ is odd. As seen in Figure~\ref{fig-double-crosscap}, we delete the edges $(y_0, \alpha)$, $(y_1, \beta)$, and $(\alpha, \beta)$ and modify the embedding near the edges incident with $\alpha$ and $\beta$, creating a large face.

\begin{figure}[!ht]
\centering
\includegraphics{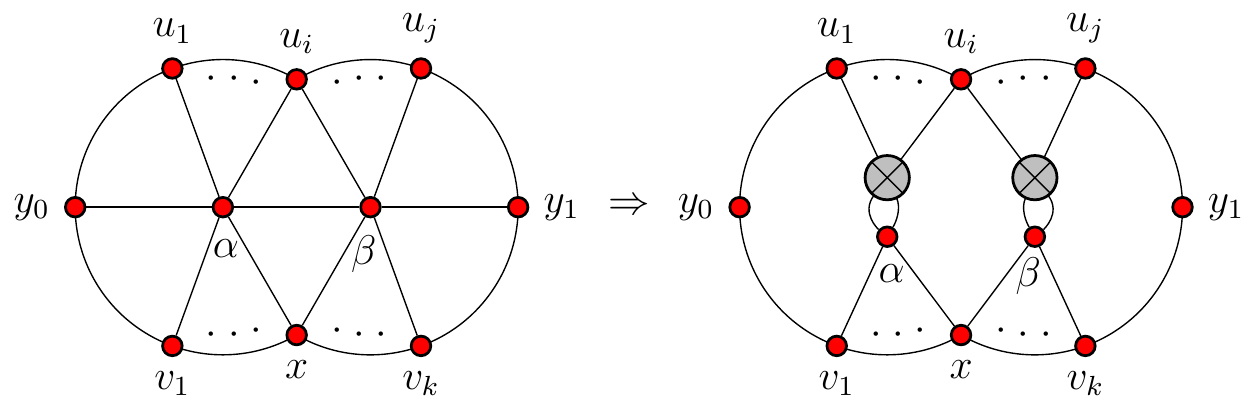}
\caption{Using two crosscaps to put $x$, $y_0$, and $y_1$ on the same face.}
\label{fig-double-crosscap}
\end{figure}

After adding the edge $(y_0,y_1)$ and contracting it, the placements of $(x,y)$, followed by $(\alpha,\beta)$, are forced, as shown in Figure~\ref{fig-doublecrosscap-face}. Depending on how we re-insert $(y_0, \alpha)$ and $(y_1, \beta)$, we get an embedding of type $(5)$, or $(4,4)$.

\begin{figure}[!ht]
\centering
\includegraphics{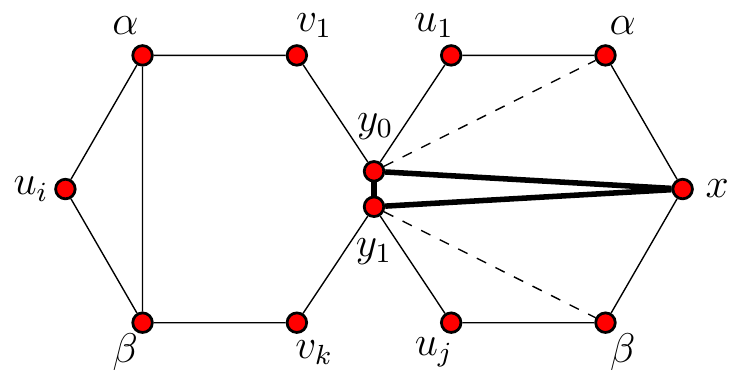}
\caption{Adding in all the missing edges after inserting two crosscaps. In this case, the choice of where to add the dashed edges yields an embedding of type $(5)$.}
\label{fig-doublecrosscap-face}
\end{figure}
\end{proof}

Case 2 featured an inductive construction where an embedding of $K_{12s+2}$ is built up from embeddings of smaller complete graphs. In particular, Youngs (see Ringel~\cite[\S10.2]{Ringel-MapColor}) proved the following:

\begin{theorem}
If there exists a nonorientable triangular embedding of $K_{2t+2}$ with two extra edges, then there exists a nonorientable triangular embedding of $K_{4t+2}$ with two extra edges. 
\label{thm-induct}
\end{theorem}

The construction for Theorem~\ref{thm-induct} takes an embedding of $K_{2t+2}$ with two extra edges and glues it, without any additional augmentation, to a triangular embedding of another graph, 
so with the same construction we can show several related statements. 

\begin{corollary}
If there exists a nonorientable embedding of type $(5)$ (resp. type $(4,4)$) of $K_{2t+2}$, then there exists a nonorientable embedding of type $(5)$ (resp. type $(4,4)$) of $K_{4t+2}$.
\label{corr-induct} 
\end{corollary}

Combining Corollary~\ref{corr-induct} with the above construction for Case 8, we obtain

\begin{theorem}
For $s \geq 1$, there exist nonorientable embeddings of type $(5)$ and $(4,4)$ of $K_{12s+2}$. 
\end{theorem}
\begin{proof}
We show by induction that such embeddings exist for $K_{6q+2}$ for $q \geq 1$. Half of the work is already done---when $q$ is odd, this is Case 8. When $q$ is even, suppose there exist embeddings of type $(5)$ and $(4,4)$ of $K_{6q'+2}$ for all $q' < q$. Then, apply Corollary~\ref{corr-induct} for $t = \frac{3}{2}q$. 
\end{proof}

Finally, we are left with the exceptional case $K_7$. Franklin~\cite{Franklin-SixColor} showed that $K_7$ does not embed in $N_2$, but the graph is embeddable in $N_3$---we simply add a crosscap along any edge to the triangular embedding of $K_7$ in the torus $S_1$. From this (cellular) embedding we obtain the other embedding types.

\begin{proposition}
$K_7$ has nonorientable embeddings of type $(6)$, $(5,4)$, and $(4,4,4)$.
\end{proposition}
\begin{proof}
Let $(a,b)$ be an arbitrary edge, and suppose the rotation at $b$ is of the form
{
\setlength{\arraycolsep}{5pt}
$$\begin{array}{rrrrrrrrrrrrrrrrrrrrrrrrrrrrrrrr}
b. & \dots & c & a & d & e & \dots
\end{array}$$
}
Adding a crosscap along the edge $(a,b)$ produces the 6-sided face
$$[a, b, c, a, b, d].$$
Exchanging the chord $(b,d)$ creates a 6-sided face with one repeated vertex $b$, as in Figure~\ref{fig-non7}. The remaining embeddings follow from the same construction as used in Proposition~\ref{prop-simple6}.

\begin{figure}
\centering
\includegraphics{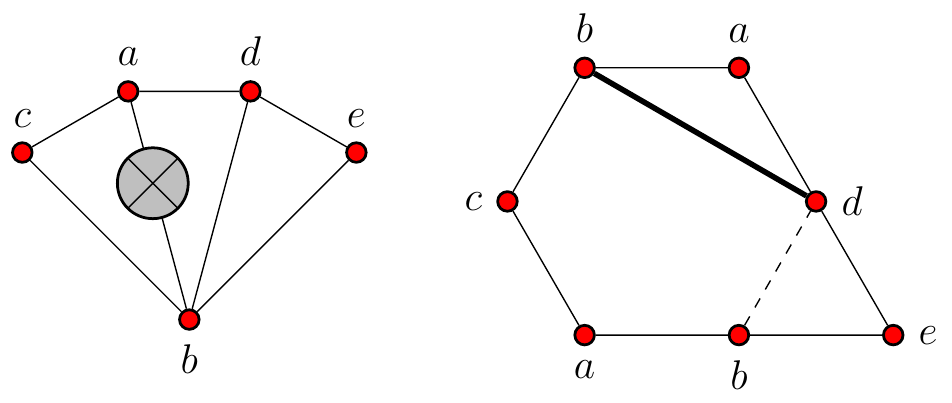}
\caption{The embedding of $K_7$ in $N_3$ is already nearly triangular. A few chord exchanges yield the remaining embedding types.}
\label{fig-non7}
\end{figure}
\end{proof}

\section{Maximum genus embeddings}\label{sec-maxgenus}

The \emph{(orientable) maximum genus} $\gamma_M(G)$ is the largest integer $g$ such that $G$ has a cellular embedding in $S_g$. Archdeacon and Craft~\cite{Archdeacon} also ask if $K_n$ has a nearly triangular maximum genus embedding. Nordhaus \emph{et al.}~\cite{Nordhaus-MaxGenus} show that $K_n$ is \emph{upper-embeddable}, meaning it has an embedding with one or two faces, depending on the parity of $|V(G)|-|E(G)|$. In particular, the maximum genus embedding has one face exactly when $n \equiv 1, 2 \pmod{4}$. The one-face embeddings are already nearly triangular in a trivial way, so we need a construction just for two-face embeddings. 

A special case of Xuong's characterization~\cite{Xuong-MaximumGenus} of maximum genus states that a graph $G$ is upper-embeddable if and only if there is a spanning tree $T$ such that $G-T$ has at most one component with an odd number of edges. To construct the one- or two-face embedding, the edges of $G-T$ are partitioned into pairs such that the edges of each pair share a vertex. Starting with an arbitrary embedding of the spanning tree $T$ in the plane (which has one face), we add the pairs one by one, as in Figure~\ref{fig-2add}. After each addition, the resulting embedding still has one face. If there is an edge left over (i.e. one of the edges of the odd-sized component), it is added arbitrarily into the embedding, resulting in a two-face embedding. 

\begin{figure}[!ht]
\centering
\includegraphics{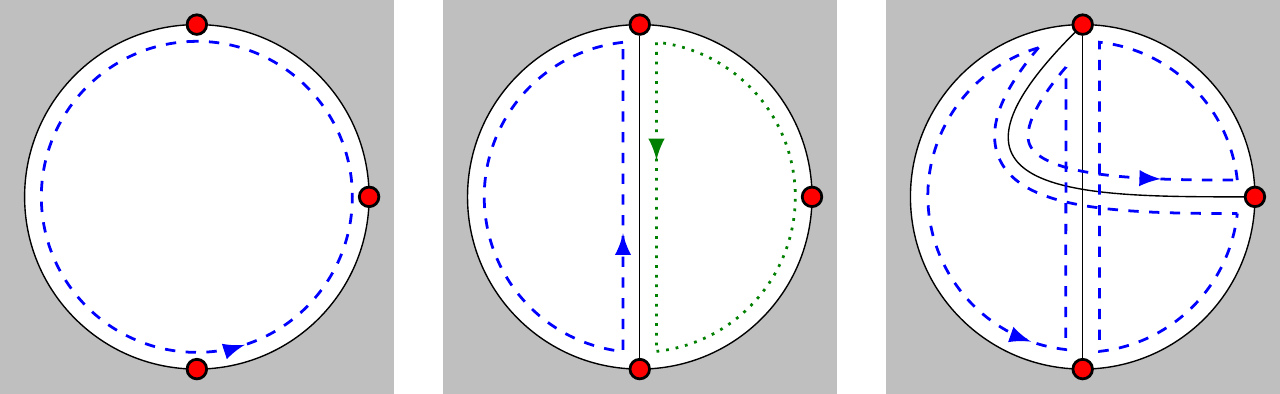}
\caption{Starting from a one-face embedding, we can add two incident edges to get another one-face embedding.}
\label{fig-2add}
\end{figure}

We note that the final embedding of $G$ restricted to $T$ is the same as the original embedding of $T$ that we started with. This observation is enough for constructing a nearly triangular two-face embedding.

\begin{proposition}
For $n \equiv 0, 3 \pmod{4}$, there exists a two-face embedding of $K_n$ where one of the faces is a triangle. 
\end{proposition}
\begin{proof}
Label the vertices $1, \dotsc, n$. Delete the edge $(2,3)$ and let the spanning tree $T$ be all the edges incident with vertex $1$. Then, $(K_n - (2,3)) - T$ is connected and has an even number of edges. Let the rotation at vertex $1$ simply be 
{
\setlength{\arraycolsep}{5pt}
$$\begin{array}{rrrrrrrrrrrrrrrrrrrrrrrrrrrrrrrr}
1. & 2 & 3 & \dotsc & n.
\end{array}$$
}
Adding in all the edge pairs in the manner described above preserves the rotation at $1$, resulting in an embedding with one face of the form $[\dotsc 2, 1, 3 \dotsc]$. We can then insert the edge $(2,3)$ into the embedding to get one triangular face $[2, 1, 3]$ and one long nontriangular face. 
\end{proof}

Finally, the problem for the nonorientable maximum genus $\overline{\gamma}_M$ is the simplest of them all. A well-known result (see, e.g., Stahl~\cite{Stahl-GeneralizedEmbedding}) states that \emph{every} connected graph has a one-face embedding in a nonorientable surface, so there is nothing to prove. We use the basic construction to prove the following ``interpolation'' theorem:

\begin{corollary}
For every nonorientable surface $N_k$, where $$k \in [\overline{\gamma}(K_n), \overline{\gamma}_M(K_n)],$$ there exists a nearly triangular embedding of $K_n$ in $N_k$. 
\end{corollary}
\begin{proof}
Let $\phi$ be a nearly triangular minimum genus embedding of $K_n$. Let $f$ be the nontriangular face in $\phi$---if the embedding is triangular, select any face arbitrarily. If $\phi$ is not already a one-face embedding, then there exists an edge incident with $f$ and a different face.\footnote{One way of seeing this is to note that $N_k \setminus V(K_n)$ is path-connected, and hence a path from some other face to $f$ in this punctured surface must intersect such an edge.} Adding a crosscap on this edge merges the two faces, incrementing the genus of the embedding. Applying this procedure repeatedly, starting from a minimum genus embedding and ending at a one-face embedding, yields the desired result.
\end{proof}

\section{Concluding Remarks}

We resolved the question of Archdeacon and Craft~\cite{Archdeacon}, classifying the complete graphs with a nearly triangular minimum genus embedding. Interest in these types of embeddings originated in searching for nonisomorphic minimum genus embeddings of the complete graph. While Korzhik and Voss~\cite{KorzhikVoss} found exponential families of embeddings for the complete graphs that do not triangulate a surface, their approach only looked at a single face distribution per graph. Can the results presented here be used to construct exponential families for the other face distributions? 

The techniques of Korzhik and Voss~\cite{KorzhikVoss} are extendable to Cases 1, 8, 10, 11 (and 9) by modifying the rotations at the hidden vertices (i.e. those replaced by the box in Figure~\ref{fig-ladder}). They also construct exponential families of nearly triangular embeddings for Case 5, so the same construction with chord exchanges should produce exponential families for the other face distributions. The situation is uncertain for Cases 2 and 6---Korzhik and Voss had constructions for these Cases, but they used different current graphs that have not been shown to lead to nearly triangular embeddings. 

For Case 6, we found a solution using a result of Gross~\cite{Gross-Case6} for the related problem of finding triangular embeddings of ``nearly complete'' graphs. Do other nearly complete graphs belonging to the other Cases have similar results? Case 9 seems to be the most accessible, since triangular embeddings of $K_{12s+9}-K_3$ can be constructed in a similar way (see Youngs~\cite{Youngs-3569}) as the embeddings of $K_{12s+6}-K_3$ used by Gross.  

The theory of index 3 current graphs~\cite[\S9]{Ringel-MapColor} allowed us to extend the additional adjacency approach of Ringel and Youngs~\cite{RingelYoungs-Case8} to handle $K_{20}$, which previously needed a purely impromptu embedding. The current graph used to generate the rows found in Section~\ref{sec-case8} is drawn in Figure~\ref{fig-index3}. Two interesting directions would be generalizing this solution for all larger graphs in Case 8, or using the same approach for Case 11. Vortices of type (T3) have a natural interpretation in index 3 current graphs, so this approach seems viable. 

\begin{figure}[!ht]
\centering
\includegraphics[scale=.85]{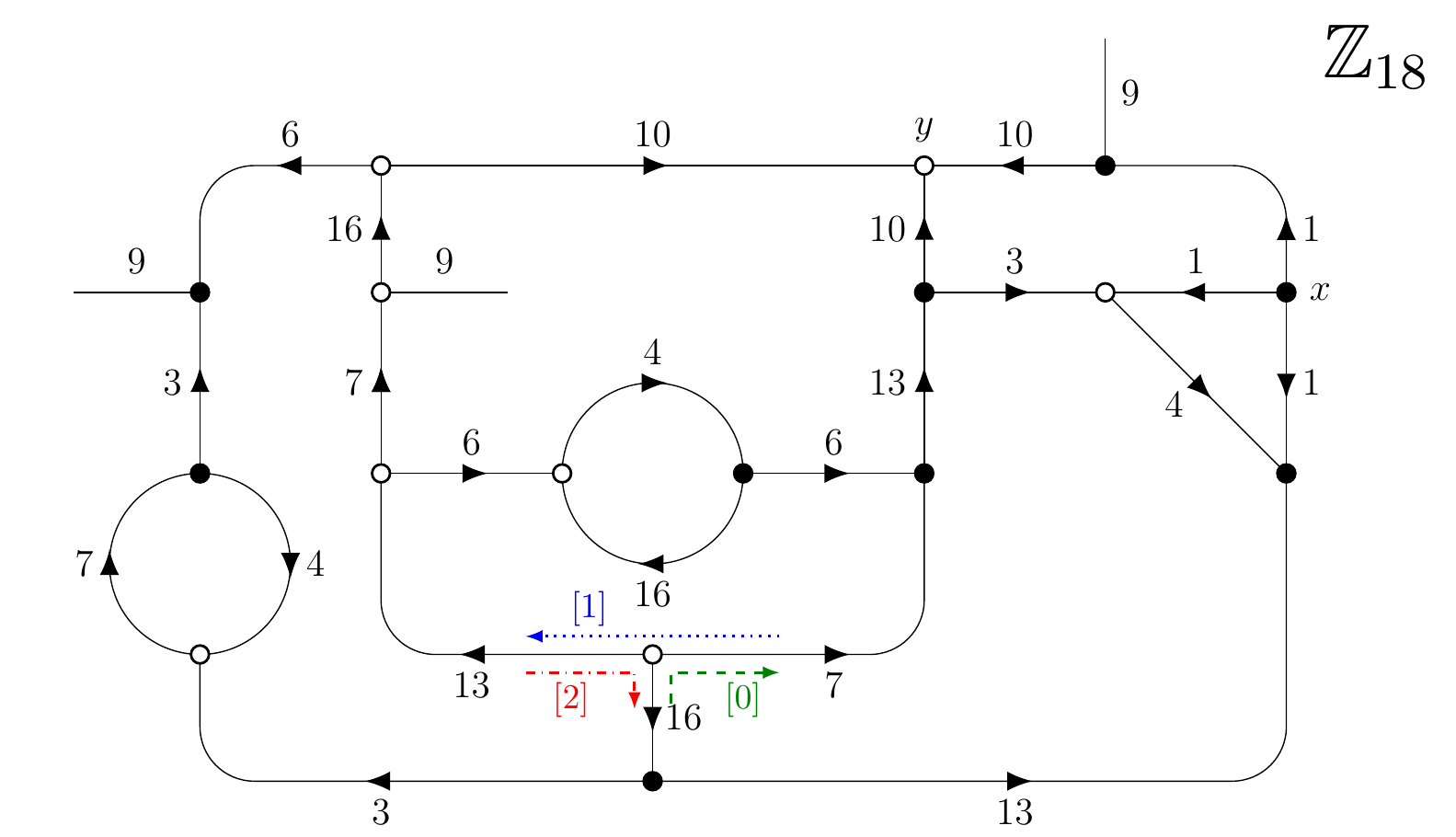}
\caption{An index 3 current graph for Case 8, $s=1$.}
\label{fig-index3}
\end{figure}

\bibliographystyle{alpha}
\bibliography{biblio}

\newpage

\appendix

\section{Embeddings of small graphs}

The following rotation systems all satisfy Rule R*. The nontriangular face of our embedding of $K_{23}$ is subdivided with a new lettered vertex to make the embedding triangular---deleting that vertex reveals the desired embedding. 

\setlength{\tabcolsep}{0.33em}

\begin{table}[ht!]
\centering
{\small
$\begin{array}{rrrrrrrrrrr}
0. & 2 & 6 & 5 & 7 & 4 & 3 & 8 & 9 \\
1. & 3 & 5 & 6 & 9 & 4 & 8 & 7 \\
2. & 0 & 9 & 7 & 5 & 8 & 4 & 6 \\
3. & 0 & 4 & 5 & 1 & 7 & 9 & 6 & 8 \\
4. & 0 & 7 & 6 & 2 & 8 & 1 & 9 & 5 & 3 \\
5. & 0 & 6 & 1 & 3 & 4 & 9 & 8 & 2 & 7 \\
6. & 0 & 2 & 4 & 7 & 8 & 3 & 9 & 1 & 5 \\
7. & 0 & 5 & 2 & 9 & 3 & 1 & 8 & 6 & 4 \\
8. & 0 & 3 & 6 & 7 & 1 & 4 & 2 & 5 & 9 \\
9. & 0 & 8 & 5 & 4 & 1 & 6 & 3 & 7 & 2 \\
\end{array}$
}
\caption{A triangular embedding of $K_{10}-P_3$.}
\label{tab-k10}
\end{table}

\begin{table}[ht!]
\centering
{\small
$\arraycolsep=3.6pt\begin{array}{rrrrrrrrrrrrrrrrrrrrrrrrrrrrrrrrrrrrrr}
1. & 23 & 19 & 12 & 17 & 6 & 9 & 2 & 7 & 18 & 20 & 8 & 5 & 16 & 14 & 3 & 11 & 22 & 21 & 15 & 13 & 4 & 10 \\
2. & 1 & 9 & 20 & 15 & 4 & 11 & 5 & 13 & 3 & 16 & 19 & 6 & 21 & 22 & 17 & 14 & 10 & 8 & 18 & 23 & 12 & 7 \\
3. & 1 & 14 & 23 & 5 & 17 & 15 & 10 & 22 & 16 & 2 & 13 & 18 & 6 & 8 & 20 & 9 & 19 & 4 & 12 & 21 & 7 & 11 \\
4. & 1 & 13 & 22 & 18 & 9 & 11 & 2 & 15 & 23 & 6 & 16 & 8 & 7 & 14 & 17 & 21 & 20 & 5 & 12 & 3 & 19 & 10 \\
5. & 1 & 8 & 12 & 4 & 20 & p & 15 & 18 & 22 & 19 & 17 & 3 & 23 & 7 & 21 & 9 & 6 & 14 & 13 & 2 & 11 & 10 & 16 \\
6. & 1 & 17 & 10 & 20 & 16 & 4 & 23 & 13 & 21 & 2 & 19 & 15 & p & 12 & 11 & 7 & 22 & 8 & 3 & 18 & 14 & 5 & 9 \\
7. & 1 & 2 & 12 & 13 & 15 & 19 & 14 & 4 & 8 & 17 & 20 & 22 & 6 & 11 & 3 & 21 & 5 & 23 & 16 & 9 & 10 & 18 \\
8. & 1 & 20 & 3 & 6 & 22 & 23 & 14 & 9 & 17 & 7 & 4 & 16 & 21 & 18 & 2 & 10 & 13 & 19 & 11 & 15 & 12 & 5 \\
9. & 1 & 6 & 5 & 21 & 13 & 17 & 8 & 14 & 22 & 12 & 23 & 10 & 7 & 16 & 15 & 11 & 4 & 18 & 19 & 3 & 20 & 2 \\
10. & 1 & 4 & 19 & 16 & 5 & 11 & 21 & 12 & 22 & 3 & 15 & 20 & 6 & 17 & 13 & 8 & 2 & 14 & 18 & 7 & 9 & 23 \\
11. & 1 & 3 & 7 & 6 & 12 & 14 & 21 & 10 & 5 & 2 & 4 & 9 & 15 & 8 & 19 & 18 & 17 & 16 & 13 & 20 & 23 & 22 \\
12. & 1 & 19 & 13 & 7 & 2 & 23 & 9 & 22 & 10 & 21 & 3 & 4 & 5 & 8 & 15 & 14 & 11 & 6 & p & 20 & 18 & 16 & 17 \\
13. & 1 & 15 & 7 & 12 & 19 & 8 & 10 & 17 & 9 & 21 & 6 & 23 & 18 & 3 & 2 & 5 & 14 & 20 & 11 & 16 & 22 & 4 \\
14. & 1 & 16 & 20 & 13 & 5 & 6 & 18 & 10 & 2 & 17 & 4 & 7 & 19 & 21 & 11 & 12 & 15 & 22 & 9 & 8 & 23 & 3 \\
15. & 1 & 21 & 23 & 4 & 2 & 20 & 10 & 3 & 17 & 22 & 14 & 12 & 8 & 11 & 9 & 16 & 18 & 5 & p & 6 & 19 & 7 & 13 \\
16. & 1 & 5 & 10 & 19 & 2 & 3 & 22 & 13 & 11 & 17 & 12 & 18 & 15 & 9 & 7 & 23 & 21 & 8 & 4 & 6 & 20 & 14 \\
17. & 1 & 12 & 16 & 11 & 18 & 21 & 4 & 14 & 2 & 22 & 15 & 3 & 5 & 19 & 23 & 20 & 7 & 8 & 9 & 13 & 10 & 6 \\
18. & 1 & 7 & 10 & 14 & 6 & 3 & 13 & 23 & 2 & 8 & 21 & 17 & 11 & 19 & 9 & 4 & 22 & 5 & 15 & 16 & 12 & 20 \\
19. & 1 & 23 & 17 & 5 & 22 & 20 & 21 & 14 & 7 & 15 & 6 & 2 & 16 & 10 & 4 & 3 & 9 & 18 & 11 & 8 & 13 & 12 \\
20. & 1 & 18 & 12 & p & 5 & 4 & 21 & 19 & 22 & 7 & 17 & 23 & 11 & 13 & 14 & 16 & 6 & 10 & 15 & 2 & 9 & 3 & 8 \\
21. & 1 & 22 & 2 & 6 & 13 & 9 & 5 & 7 & 3 & 12 & 10 & 11 & 14 & 19 & 20 & 4 & 17 & 18 & 8 & 16 & 23 & 15 \\
22. & 1 & 11 & 23 & 8 & 6 & 7 & 20 & 19 & 5 & 18 & 4 & 13 & 16 & 3 & 10 & 12 & 9 & 14 & 15 & 17 & 2 & 21 \\
23. & 1 & 10 & 9 & 12 & 2 & 18 & 13 & 6 & 4 & 15 & 21 & 16 & 7 & 5 & 3 & 14 & 8 & 22 & 11 & 20 & 17 & 19 \\
p. & 6 & 15 & 5 & 20 & 12
\end{array}$
}
\caption{An embedding of type $(5)$ of $K_{23}$. }
\label{tab-k23}
\end{table}

\section{A new solution for Case 1}\label{app-c1}

In Section~\ref{sec-c1}, we found nearly triangular embeddings of $K_{12s+1}$ for $s \geq 3$ using a single family of current graphs with a large simple zigzag. In fact, there even exist families of current graphs for $K_{12s+1}-K_3$ that include the $s = 2$ case as well, such as the one in Figure~\ref{fig-sgen}. This is the simplest known proof of Case 1 for $s \geq 2$ of the original Map Color Theorem, and as remarked by Ringel~\cite[p.96]{Ringel-MapColor}, there cannot exist an index 1 current graph with three vortices of type (T1) for $s=1$ because $\Z_{10}$ does not have enough generators. 

\begin{figure}[!ht]
    \centering
    \begin{subfigure}[b]{0.99\textwidth}
    \centering
        \includegraphics[scale=0.9]{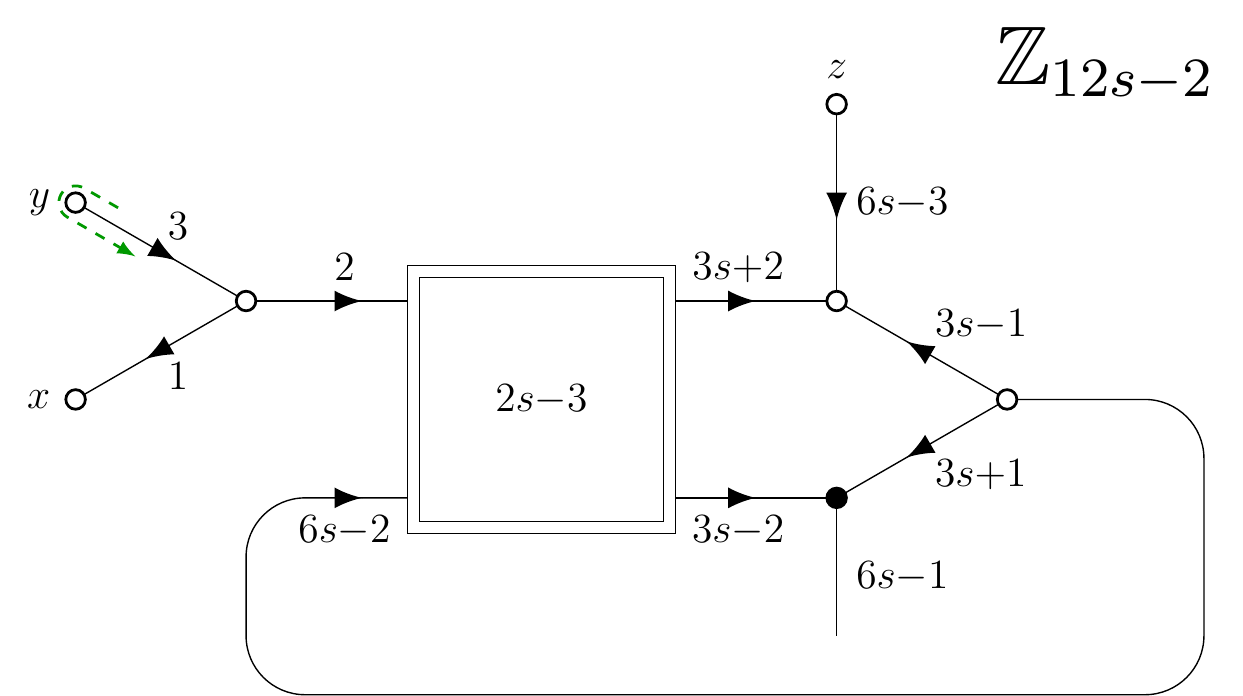}
        \caption{}
        \label{subfig-c1a}
    \end{subfigure}
    \begin{subfigure}[b]{0.99\textwidth}
    \centering
        \includegraphics[scale=0.8]{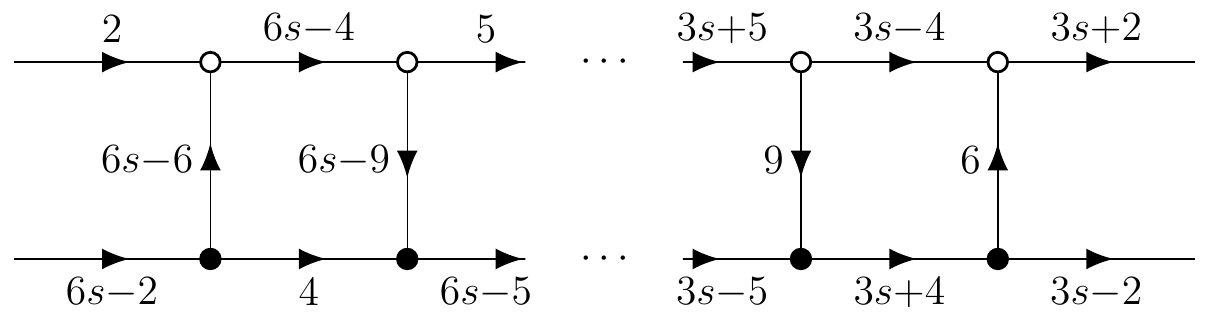}
        \caption{}
        \label{subfig-c1b}
    \end{subfigure}
\caption{Current graphs producing triangular embeddings of $K_{12s+1}-K_3$ for all $s \geq 2$.}
\label{fig-sgen}
\end{figure}

\end{document}